\documentclass[11pt]{article}
\pdfoutput=1

\usepackage[margin=1.05in]{geometry}
\usepackage{amsmath,amssymb,amsthm}
\usepackage{mathtools}
\usepackage{graphicx}
\usepackage{booktabs}
\usepackage{placeins}
\usepackage{needspace}
\usepackage{array}
\usepackage{microtype}
\usepackage{xcolor}
\usepackage[shortlabels]{enumitem}
\setlist{topsep=3pt,itemsep=1pt,parsep=0pt}
\usepackage[colorlinks=true,linkcolor=blue!45!black,citecolor=blue!45!black,urlcolor=blue!45!black]{hyperref}
\hypersetup{pdftitle={Uniform Bessel bounds for magnetization crossings in finite Ising chains},pdfauthor={Arjun Pemmasani}}

\newtheorem{theorem}{Theorem}[section]
\newtheorem{lemma}[theorem]{Lemma}
\newtheorem{proposition}[theorem]{Proposition}
\newtheorem{corollary}[theorem]{Corollary}
\theoremstyle{definition}

\theoremstyle{remark}
\newtheorem{remark}[theorem]{Remark}

\newcommand{\E}{\mathbb{E}}
\newcommand{\Prob}{\mathbb{P}}
\newcommand{\Var}{\operatorname{Var}}

\newcommand{\chg}{\operatorname{chg}}
\newcommand{\rep}{\operatorname{rep}}
\newcommand{\sR}{\mathord{+}}
\newcommand{\sL}{\mathord{-}}
\newcommand{\cE}{\mathcal E}
\def\zenodoDOI{10.5281/zenodo.23417554}

\title{Uniform Bessel bounds for magnetization crossings\\ in finite Ising chains}
\author{Arjun Pemmasani\\[2pt]
\small Department of Mathematics, Harvey Mudd College\\
\small Claremont, CA 91711, USA\\
\small \texttt{apemmasani@hmc.edu}}
\date{}
\begin{document}
\maketitle

\begin{abstract}
We consider a finite Ising chain in zero field and ask when the all-up
configuration is as likely as zero magnetization (magnetization one for
an odd number of spins). For both free and periodic ends, the crossing
coupling of a long chain is half the logarithm of its length minus half
the logarithm of that logarithm, up to a term that vanishes as the
length grows. We compute further terms with explicit constants and, for
free ends, bound the crossing at every length. The second term is
missing from an earlier approximation and explains why an earlier numerical
fit gives a slope below one half. To locate the crossing, we compare the
known magnetization law with its Bessel approximation. The Bessel
approximation is above the exact weight, and we bound its relative error
for every length, coupling and magnetization. At the crossing this bound
tends to zero. The central bin still has about as many domain walls as
the logarithm of the length on average, so the crossing has a coupling
width of order the inverse logarithm. In a fixed positive dimensionless
field the last crossing comes from near the edge of the row, at a
coupling of about a quarter of the logarithm of the length. For long
chains it is the one-spin crossing when the field exceeds half the
logarithm of two. Below that value another bin crosses later.
\end{abstract}



\section{Introduction}\label{sec:intro}

\subsection{The crossing}

Let $Z_N(M)$ be the sum of the Ising weights of the configurations of $N$
spins with magnetization $M$ (the partition function of the fixed-magnetization
ensemble). Counting configurations by their domain walls gives a
classical finite sum~\cite{ising1925,antal2004,garciapelayo2009,dantchev2022}. When the
mean number of walls stays bounded as $N$ grows, this sum has a modified
Bessel approximation~\cite{antal2004,dantchev2023,dantchev2024}. We bound
its relative error explicitly for every length, coupling and
magnetization (Theorem~\ref{thm:allbin-bessel} and
Corollary~\ref{cor:fixed-M} for free ends, Theorem~\ref{thm:periodic-bessel} for periodic
ends). The Bessel form always overestimates the
exact weight. Its error also stays small when the mean number of walls
grows like $\log N$, which is the regime of the crossing we study.

Take $N\ge2$ spins $\sigma_i=\pm1$ with free ends and zero-field Ising
weight $\exp(J\sum_{i<N}\sigma_i\sigma_{i+1})$, where $J>0$. Here $J$ is
dimensionless, and it is the product $\beta J_{\rm ex}$ of the inverse temperature and
the exchange coupling $J_{\rm ex}$ of~\cite{stepanyan2024}. The all-up
configuration is the only one with magnetization $\sigma_1+\dots+\sigma_N=N$.
Zero magnetization (or magnetization $\pm1$ for odd $N$) comes from many configurations. For small $J$ domain walls cost
little and a central magnetization is the more likely of the two, and for large $J$
the all-up configuration is. They are equally likely at a single coupling,
the \emph{central crossing}. Stepanyan et~al.\@~\cite{stepanyan2024} study
this crossing in the periodic chain as a thermal transition of a finite
system (the analogue of the transition temperature of a first-order
transition). They approximate it by $J\approx(2\log N-1)/5$. Like them, we compare one
central bin with the single all-up configuration (Remark~\ref{rem:convention}
says what changes for other comparisons). We prove the asymptotics of the crossing with
explicit constants for free and for periodic ends, and for free ends we bound
it at every $N$.

Read from left to right, the spins are the steps of a persistent random
walk~\cite{goldstein1951,gillis1955}, a correspondence that
Garc\'{\i}a-Pelayo~\cite{garciapelayo2009} uses to derive the exact
magnetization laws of the free and the periodic chain, also in a field. The first step is fair, and
each later step repeats the previous one with probability $p$, where
$e^{-2J}=(1-p)/p$. The magnetization is the position after $N$ steps, so the
walk that never turns, which reaches an end, competes with the many walks
that end in the middle.

With two steps the walk ends in the middle exactly when it turns once, so the
three bins have probabilities $\frac p2$, $1-p$ and $\frac p2$, and they are
level at $p=\frac23$. With three steps the bins have probabilities
$\frac{p^2}2$, $\frac{1-p^2}2$, $\frac{1-p^2}2$ and $\frac{p^2}2$, and they are
level at $p=1/\sqrt2$. In terms of the coupling these two crossings are at
$J=\frac12\log2$ and $J=\frac12\log(1+\sqrt2)$. In general exactly one walk reaches a given endpoint,
namely the walk that never turns, and very many walks reach the middle by
turning a few times (each turn brings a factor of about $1-p$).

The middle bin and the endpoints are level when the walk turns about $\log N$
times on average. More precisely, the crossing value $p_N$ of $p$ satisfies
$N(1-p_N)=\log N-\frac12\log\log N-0.467\ldots+o(1)$, and for every $N\ge2$
we trap $N(1-p_N)$ between two explicit functions of $N$ whose distance tends
to $0$ (Theorem~\ref{thm:every-n}). For $p<p_N$ the middle is more likely, and
for $p>p_N$ the ends are. In terms of the coupling, the crossing is at
$J=\frac12(\log N-\log\log N)+O(\log\log N/\log N)$ for free and for periodic
ends, and the correlation length there is about $N/(2\log N)$
(Corollary~\ref{cor:couplings}). Their approximation and their fitted line~\cite{stepanyan2024}
have no $\log\log N$ term, and this term explains why the slope of the
fit is below $\frac12$ (Section~\ref{sec:coupling}). At the
free crossing the central bin carries about $\log N$ domain walls, so its mean energy
exceeds that of the all-up configuration by only $2\log N-\log\log N+O(1)$
in units of the exchange coupling (Proposition~\ref{prop:walls}). For this zero-field comparison of one
central bin with all up, the crossing therefore has a width of order
$1/\log N$ in the coupling. The width $1/N$ of a first-order transition
assumes a nonzero latent heat per spin, so that the energy gap grows like the volume.

The crossing lies between two regimes in which the law is well understood.
When $p=1-\tau/N$ with $\tau>0$ fixed, the rescaled walk tends to the integrated
telegraph process and the magnetization per spin has a scaling
limit~\cite[Theorem~2.1 and Remark~3.4]{herrmann2010}. When $0<p<1$ is fixed,
the magnetization has a normal limit (Theorem~\ref{thm:clt}). At the crossing $\tau$ grows
like $\log N$, so neither limit decides it, and we need bounds that hold
uniformly in $\tau$. Our main tool is a comparison of the middle bin with a
modified Bessel function, with a relative error that is explicit for every $N$
(Theorem~\ref{thm:bessel}). A walk with $d$ directions, which
is the $d$-state Potts chain, is treated in~\cite{paperB}.

\subsection{Setting and notation}\label{sec:setting}

\paragraph{The walk.} The walk makes $N\ge1$ steps. We record a path of the walk as a word
$w=w_1\cdots w_N$ in the letters $\sR$ (right) and $\sL$ (left). The letter $w_1$
is $\sR$ or $\sL$ with probability $\frac12$, and given $w_1,\dots,w_i$,
\[
\Prob(w_{i+1}=w_i)=p,\qquad \Prob(w_{i+1}\ne w_i)=q=1-p .
\]
An index $i<N$ with $w_i\ne w_{i+1}$ is a \emph{switch}. Let $\chg(w)$ be the
number of switches and $\rep(w)=N-1-\chg(w)$. Then
\begin{equation}\label{eq:wordprob}
\Prob(w)=\tfrac12\,p^{\rep(w)}q^{\chg(w)} .
\end{equation}
For $p>0$ we put $u=q/p$, the odds of a switch.

\paragraph{Bins.} The walk ends in bin $k$ if $w$ has $k$ letters $\sR$, for
$0\le k\le N$. We write $P_N(k)$ for the probability of bin $k$, and we set $P_0(0)=1$.
We call the bins $0,\dots,N$ together the \emph{row} (as in a Galton board), and bins
$0$ and $N$ are its \emph{endpoints}. Only the constant word
$\sL\cdots\sL$ reaches bin $0$, and only $\sR\cdots\sR$ reaches bin $N$, so
\[
P_N(0)=P_N(N)=\tfrac12p^{N-1}\qquad(N\ge1).
\]
We call bin $m_N=\lfloor N/2\rfloor$ a middle bin. For odd $N$ the two middle
bins $m_N$ and $m_N+1$ have equal probability by symmetry.

\paragraph{Spins.} Put $\sigma_i=+1$ if $w_i=\sR$ and $\sigma_i=-1$ if
$w_i=\sL$, and $X_N=\sigma_1+\dots+\sigma_N$. The bin is $K_N=(N+X_N)/2$. Since
$\chg(w)=\sum_{i<N}(1-\sigma_i\sigma_{i+1})/2$, \eqref{eq:wordprob}
shows that $\Prob(w)$ is proportional to
$\exp(J\sum_{i<N}\sigma_i\sigma_{i+1})$ with $e^{-2J}=u$. So the
spins form the Ising chain in zero field with free ends, and $J>0$
exactly when $p>\frac12$.

\paragraph{Scales.} Let $L=\log N$ and $\tau=Nq$. Each of the $N-1$ steps after
the first switches with probability $q$, so the expected number of switches is
$(N-1)q=\tau-q$. We call $\tau$ the \emph{switching scale}. In the Ising chain a switch is a domain wall, and
$q=1/(1+e^{2J})$, so $\tau$ is the expected number of domain walls plus $q$. The central crossing has
$\tau$ of order $L$. Near the edges the crossings have larger $\tau$, up to order
$\sqrt N$ for the bins next to an edge (Section~\ref{sec:allbin}).
We use several equivalent parameters, namely the persistence $p$, the odds
$u=q/p=e^{-2J}$ of a switch, the coupling $J$, the switching scale $\tau=Nq$,
and in the Bessel estimates $w=Nu$ and $x=w/2$ (from Section~\ref{sec:middle} on,
the context distinguishes it from a word $w$). For a general bin $a$,
Sections~\ref{sec:allbin-bessel} and~\ref{sec:periodic} instead put
$x=\sqrt{ab}\,u$, which equals $w/2$ at the middle bin when $N$ is even. Near the central crossing
$u$ is of order $L/N$, and then $\tau$ and $Nu=2x$ differ by $O(L^2/N)$.

\paragraph{Ratios and crossings.} For a bin $k$ with $1\le k\le N-1$ we write
$a=k$ and $b=N-k$ for its numbers of letters $\sR$ and $\sL$. For $p>0$ the
ratio of bin $k$ to an endpoint depends on $p$ only through $u$, and it is a
polynomial in $u$ with nonnegative integer coefficients
(see \eqref{eq:allbin-polynomial} below). We write
\[
S_{a,b}(u)=\frac{P_N(a)}{P_N(0)},\qquad S_N(u)=S_{m_N,\,N-m_N}(u)
\]
for this ratio and for its value at the middle. \emph{The crossing} of bin $k$
is the value of $p$ at which bin $k$ is exactly as likely as an endpoint, that
is, $S_{a,b}(u)=1$. It is unique (Section~\ref{sec:allbin-order}), and we write
$p_{a,b}$ for it and $u_{a,b}=(1-p_{a,b})/p_{a,b}$ for its odds. The
\emph{central crossing} is $p_N=p_{m_N,N-m_N}$, with odds $u_N$, and
\begin{equation}\label{eq:zN}
z_N=N(1-p_N)=\frac{Nu_N}{1+u_N}
\end{equation}
is the value of $\tau$ there.

\paragraph{Special functions and asymptotic notation.} The functions $I_0$ and
$I_1$ are the modified Bessel functions of the first kind, of orders zero and
one. We use them through their power
series and their behavior at infinity~\cite[Eqs.~10.25.2, 10.40.1]{dlmf},
\[
I_0(2x)=\sum_{r\ge0}\frac{x^{2r}}{(r!)^2},\quad
I_1(2x)=\sum_{r\ge0}\frac{x^{2r+1}}{r!\,(r+1)!},\quad
I_\nu(z)=\frac{e^z}{\sqrt{2\pi z}}\Bigl(1-\frac{4\nu^2-1}{8z}+O(z^{-2})\Bigr)
\]
as $z\to+\infty$ along the positive real axis, for $\nu=0,1$.
We write $f\sim g$ if $f/g\to1$. Each $O(\cdot)$ comes with the parameters it
is uniform in, and $O_a(\cdot)$ means that the constant may depend on $a$.

\paragraph{Ising dictionary.} Sections~\ref{sec:exact} to~\ref{sec:allbin}
are written in the variables of the walk, and they translate as follows.
Bin $k$ is magnetization $X_N=2k-N$, so the endpoints are full magnetization
$\pm N$. A middle bin has magnetization zero for even $N$ and either $-1$
or $1$ for odd $N$. A switch is a domain wall, and its odds $u=e^{-2J}$ are the
Boltzmann factor of one wall. For free ends, $Z_N(M)$ is the sum of the weights
$\exp(J\sum_{i<N}\sigma_i\sigma_{i+1})$ of the configurations with
magnetization $M$, so $Z_N(N)=e^{(N-1)J}$ and
$S_{a,b}(u)=Z_N(a-b)/Z_N(N)$. The crossing of bin $k$ is the coupling
$-\frac12\log u_{a,b}$ at which magnetization $a-b$ and
magnetization $N$ are equally likely, and at the center it is
$J_N^{\mathrm F}=-\frac12\log u_N$. The correlation length of the infinite
chain in zero field, with no constraint on the magnetization, is
$\xi(J)=-1/\log\tanh J$, and $\tanh J=(1-u)/(1+u)$, so
$N/\xi=N\log\frac{1+u}{1-u}=2Nu(1+O(u^2))$. Since $\tau=Nu/(1+u)$, this gives
$\tau=\frac N{2\xi}(1+O(u))$ as $u\to0$. To leading order, the switching scale is half the
length of the chain measured in correlation lengths. It is exactly the
scaling variable $\zeta=N(1-\tanh J)/2$ of Antal, Droz and
R\'acz~\cite{antal2004} (we keep the letter $\zeta$ for $\zeta_N$ below).

\subsection{Main results}

\paragraph{The coupling at the crossing.} Let $J_N^{\mathrm P}$ be the coupling at the periodic central crossing, and
$J_N^{\mathrm F}$ the free one of Section~\ref{sec:setting}. Put
$c_0=\frac12\log(\pi/8)=-0.467\ldots$ and $c_{\mathrm P}=\frac12\log(\pi/2)$.
Theorem~\ref{thm:periodic-crossing} and Corollary~\ref{cor:couplings} give
\[
J_N^X=\tfrac12(\log N-\log\log N)+\frac{\frac14\log\log N-\frac12c_X}{\log N}
+O\Bigl(\frac{(\log\log N)^2}{(\log N)^2}\Bigr)
\]
for free ends ($X=\mathrm F$, with $c_{\mathrm F}=c_0$) and for periodic ends
($X=\mathrm P$). Since
$c_{\mathrm P}=c_0+\log2$, periodic ends lower the coupling at the crossing by
$(\log2)/(2\log N)$ to leading order. At either crossing the correlation
length is about $N/(2\log N)$, a vanishing fraction of the chain. The
approximation $J\approx(2\log N-1)/5$ of~\cite[Eq.~(16)]{stepanyan2024}
grows too slowly. At $N=10^5$ the exact periodic coupling is $4.576$. The
periodic case of the display without its error term gives $4.578$, and the approximation gives
$4.405$ (Section~\ref{sec:coupling}).

\paragraph{A Bessel bound at every $N$.} Put $h=N/2$ and
$F(x)=2x(I_0(2x)+I_1(2x))$. Theorem~\ref{thm:bessel} says that for every
$N\ge2$ and every $x>0$,
\[
0\le1-\frac{hS_N(x/h)}{F(x)}\le\frac{x(x+1)}h .
\]
Near the crossing $x=Nu/2$ is about $L/2$, so the relative error is
$O(L^2/N)$. Theorem~\ref{thm:allbin-bessel} gives the
same kind of bound for every bin, with relative error at most $Nu^2/2+u$.

\paragraph{The central crossing at every $N$.} For free ends, let $\zeta_N$ be the positive
root of $8\zeta e^{2\zeta}=\pi N^2$, that is, $\zeta_N=\frac12W_0(\pi N^2/4)$
with $W_0$ the principal branch of the Lambert function, and let
$\varepsilon_N=\zeta_N(\zeta_N+1)/N+1/(16\zeta_N^2)$. Theorem~\ref{thm:every-n}
says that
\[
\zeta_N+\frac1{8\zeta_N}-\varepsilon_N<z_N<\zeta_N+\frac1{8\zeta_N}
\qquad\text{for every }N\ge2.
\]
Note that $\varepsilon_N\to0$. In terms of the coupling, with
$Y_N=\zeta_N+\frac1{8\zeta_N}$ and $Y_N'=Y_N-\varepsilon_N$, this gives
$\frac12\log(N/Y_N-1)<J_N^{\mathrm F}<\frac12\log(N/Y_N'-1)$
(Corollary~\ref{cor:every-n}(a)). For $N\ge175$ the proof is analytic. For
$2\le N\le174$ we decide the sign of $S_N-1$ at two rationals, one on each side
of the root, in exact integer arithmetic, compare the resulting bracket with
$\zeta_N$ in outward interval arithmetic, and check both steps again in the
Lean proof assistant. The same bounds give two simpler statements. With $c_0$ as above, the
crossing satisfies
\[
z_N+\tfrac12\log z_N=L+c_0+\frac1{8z_N}+R_N,\qquad
|R_N|\le\frac1{4L^2}+\frac{L^2+L}N\quad(N\ge2)
\]
(Proposition~\ref{prop:crossing-eq}), and as $N\to\infty$
\[
z_N=L-\tfrac12\log L+c_0+\frac{\frac14\log L-\frac12c_0+\frac18}{L}
+O\Bigl(\frac{(\log L)^2}{L^2}\Bigr)
\]
(Corollary~\ref{cor:every-n}(c)). The convergence is slow, and $z_N/L$ is still
only $0.84$ at $N=10^4$ (Section~\ref{sec:middle}).

\paragraph{Every bin.} For free ends, let bin $a$ be at distance $a$ from the nearer edge,
with $b=N-a\ge a$. Next to the edge the crossing is at
$p_N^{\mathrm{adj}}=(1+\sqrt{N-1})/(2+\sqrt{N-1})$. For $N\ge4$ the two bins
next to the edges are the least likely of all at the central crossing
(Proposition~\ref{prop:adjacent}). As $a\to\infty$,
Theorem~\ref{thm:allbin-expansion} expands $2\sqrt{ab}\,u_{a,b}$ to the same
order as $z_N$, uniformly in $b\ge a$. For fixed $a$ the crossing comes
instead from a Laguerre polynomial, and $u_{a,b}=c_a/\sqrt b-1/b+O_a(b^{-3/2})$ as $b\to\infty$,
where $c_a^2$ is the positive root of an explicit polynomial of degree $a$
(Theorem~\ref{thm:edge-laguerre}). The two regimes match
(Theorem~\ref{thm:edge-matching}).

\paragraph{Periodic boundaries.} If the steps are closed into a loop, so that
the last step also interacts with the first, the central crossing moves by
$(\log2)/N$ in the odds to leading order (Theorem~\ref{thm:periodic-crossing}).
Under either boundary, at odds $u_N^X+t/N$, where $u_N^X$ is the central
crossing of that boundary, the ratio of the bin nearest to $N/2+yN/\sqrt{\log N}$ to an
endpoint tends to $e^{t-2y^2}$ as $N\to\infty$, for fixed $t$ and $y$
(Theorem~\ref{thm:critical-window}).

\paragraph{Domain walls and the energy gap.}
At the free central crossing, the central bin has
$z_N+\frac12+O(1/z_N)$ domain walls on average, so its mean energy
exceeds that of the all-up configuration by
$2\log N-\log\log N+O(1)$ in units of the exchange coupling
(Proposition~\ref{prop:walls}).
Hence a factor $e$ change in their probability ratio takes a coupling
change of about $1/(2\log N)$, whose slow decay we discuss in
Section~\ref{sec:discussion}.

\paragraph{A magnetic field.} Fix a dimensionless field
$\mathcal B=\beta H>0$ (written $\beta h$ in~\cite{stepanyan2024}). For free and for periodic ends,
both comparisons behind the transitions of~\cite{stepanyan2024} then come from the edge of
the row, at $J=\frac14\log N+O(1)$, and we find the limit of the constant. For large $N$ and $\mathcal B>\mathcal B_c=\frac12\log2$, the
one-spin crossing is the last crossing of the all-up configuration as $J$ increases. Below
$\mathcal B_c$ another bin crosses later, and at $\mathcal B_c$ the
gap between the two crossing couplings tends to zero (Theorem~\ref{thm:field}).

We prove these results in Sections~\ref{sec:exact} to~\ref{sec:periodicwindow}
and discuss them in Section~\ref{sec:discussion}.
Section~\ref{sec:verify} lists what was computed and what is in
Lean, and Appendix~\ref{app:every-n} has the longer proofs of Section~\ref{sec:middle}.

\subsection{Related work}

Early treatments of the persistent (or correlated) random walk include Goldstein~\cite{goldstein1951}
and Gillis~\cite{gillis1955}. Renshaw and Henderson~\cite{renshaw1981} study the
symmetric correlated walk, derive exact transition probabilities and
investigate limiting distributions. Konno~\cite[Lemma~2 and Theorem~7]{konno2009}
writes the exact law with the hypergeometric function ${}_2F_1$ and proves the
integrated telegraph limit when $p=1-\theta/N$ with $\theta\in(0,1)$ fixed.
Kiumi, Konno and Tamura~\cite[Proposition~3]{kiumi2022} write the return
probability of the correlated random walk with Legendre polynomials.
Herrmann and Vallois~\cite[Theorem~2.1 and Remark~3.4]{herrmann2010}
give the process limit and its law, including the Bessel density and endpoint
atoms. With $\tau$ replaced by $Nu$, their law gives the Bessel form of every
bin that we use in Section~\ref{sec:allbin-bessel}. These fixed-scale limits
do not by themselves give the uniform finite-$N$ bounds needed when $\tau$
grows like $\log N$. Dekking and Kong~\cite{dekkingkong2011} study the Markov
binomial law, give its variance and find when it has one, two or three modes.
Their results already imply the crossing of the bins next to the edges and the
ordering of the bin probabilities within a row (Section~\ref{sec:allbin-order}).

For the Ising chain, Antal, Droz and R\'acz~\cite{antal2004} give the
magnetization law at a fixed switching scale $\tau$ under several
boundary conditions. Denisov and H\"anggi~\cite[Eq.~(3.6)]{denisov2005}
give the joint law of the numbers of up spins, up domains and domain walls
in the free chain. Garc\'{\i}a-Pelayo~\cite[Eqs.~(14) and (15)]{garciapelayo2009}
gives the exact free and periodic magnetization laws in hypergeometric
form, and his Eq.~(17) adds a field for both boundary conditions.
Dantchev and Rudnick~\cite{dantchev2022} give the
exact free and periodic coefficients. Dantchev, Tonchev and Rudnick
approximate these coefficients by Bessel functions when their variable
$x_t=2Nu$ is of order one, for periodic
ends~\cite[Section~4.3.1 and Appendix~D]{dantchev2023} and with relative
error $O(1/N)$ for free
ends~\cite[Section~3, Eq.~(3.10)]{dantchev2024}. For periodic ends they
also give the form of their scaling function for large
$x_t$~\cite[Eq.~(4.14)]{dantchev2023}. At our crossing,
however, $x_t$ grows like $2\log N$, so we need a remainder that is
uniform for the exact chain beyond that fixed scale.

The scaling variable of Antal, Droz and R\'acz is our $\tau$. Their law
already suggests the crossing if we evaluate it at $\tau$ of order
$\log N$ (although their limit keeps $\tau$ fixed). For free ends the
endpoint atom is $e^{-\tau}/2$, while the density at zero magnetization
is $\frac12\tau e^{-\tau}(I_0(\tau)+I_1(\tau))$
by~\cite[Eq.~(47)]{antal2004}. Note that neighboring magnetizations per
spin are $2/N$ apart. Hence the middle bin and an endpoint should balance
when the atom equals $2/N$ times this density, giving
\[
\tau\bigl(I_0(\tau)+I_1(\tau)\bigr)=N/2.
\]
This is the equation $F(x_B)=h$ for the Bessel root $x_B$ of
Theorem~\ref{thm:bessel}(ii), with $\tau=2x_B$.
Writing $z_B=2x_B$, its large-argument expansion gives
\[
z_B+\tfrac12\log z_B
=L+c_0+\frac1{8z_B}+O(z_B^{-2}),
\]
with the same displayed terms as the crossing equation of
Proposition~\ref{prop:crossing-eq}. We prove that the Bessel root
approximates the actual crossing to relative error $O(L/N)$ in the
odds (Theorem~\ref{thm:bessel}(iii)), give explicit finite-$N$ bounds, and prove the bracket of
Theorem~\ref{thm:every-n}. The known Bessel laws also suggest the
$\log\log N$ term, the periodic shift by $\log2$, and the Gaussian
profile of Theorem~\ref{thm:critical-window}. Our bounds turn these heuristics into theorems about the exact crossings.

Stepanyan et~al.\@~\cite[Section~III.B]{stepanyan2024} study two
crossings of the periodic chain as thermal transitions. Their first
crossing compares the all-up configuration with magnetization $N-2$
(one flipped spin), and is exactly at $J=\frac14\log N$. Their second
compares all up with zero magnetization, and they estimate its
coupling. We give the second crossing with its constants for both
boundary conditions. For free ends we also bound it at every $N$ and
find the crossing of every bin. The Markov binomial estimates of
\v{C}ekanavi\v{c}ius and Mikalauskas assume ranges of transition
probabilities that exclude $p\to1$ in this symmetric model
(see~\cite[Theorems~1.1 to~1.3, pp.~220--221]{cekanavicius2001}
and~\cite[Theorems~1.1 to~1.3, pp.~308--309]{cekanavicius2001ld}).
Jensen~\cite[Section~2]{jensen2013} studies the relative error of a
saddlepoint approximation to the Markov binomial law, after removing the
exact endpoint masses. His stationary chain is our free chain with $n=N$,
$\alpha=q$ and $\beta=p$. For fixed $\alpha$ and $\beta$ the relative error
is $o(1/n)$ in a large deviation region~\cite[Proposition~1]{jensen2013},
and he also studies the boundary of the parameter space, including the
limit $\alpha\to0$ with $\alpha/\beta\to0$~\cite[Proposition~2 and
Eq.~(2.6)]{jensen2013}. These results do not give signed Bessel bounds or
an error estimate that tends to zero uniformly along $Nq\sim\log N$.

We prove relative
Bessel bounds that are explicit in $N$, in $u$ and in the bin
(Theorems~\ref{thm:bessel}, \ref{thm:allbin-bessel}
and~\ref{thm:periodic-bessel}), and that stay small when the mean number of
domain walls grows like $\log N$. We bracket the free crossing at every $N$
(Theorem~\ref{thm:every-n}). We expand the actual crossings at every growing distance $a$ from the edge,
uniformly over $b\ge a$ (Theorem~\ref{thm:allbin-expansion}), describe
the crossing window about the exact root (Theorem~\ref{thm:critical-window}),
and locate both crossings in a fixed field (Theorem~\ref{thm:field}).

\section{The exact law}\label{sec:exact}

The exact law and recurrence are classical, and so are the generating
function~\cite[Eq.~(3.1)]{gillis1955} and the fixed-$p$ normal
limit~\cite[Theorem~6]{konno2009}. The count of Theorem~\ref{thm:pmf} is
Ising's count of domains~\cite{ising1925}, and it appears
in~\cite[Eq.~(44)]{antal2004} and~\cite[Eqs.~(A15) to (A18)]{dantchev2022}.
The law also appears in~\cite[Proposition~4.1 and Eq.~(4.1)]{dekkingkong2011}, and
Garc\'{\i}a-Pelayo~\cite[Eq.~(14)]{garciapelayo2009} writes it in terms of ${}_2F_1$.
Denisov and H\"anggi~\cite[Eq.~(3.6)]{denisov2005} give the joint law of the
numbers of up spins, up domains and domain walls, and Kiumi, Konno and
Tamura~\cite[Proposition~3]{kiumi2022} write the return probability of the
walk, which is the middle bin for even $N$, with Legendre polynomials.
We prove Theorem~\ref{thm:pmf} because the later sections use its count of
words by switches. By \eqref{eq:wordprob}, a word with $j$ switches has
probability $\frac12p^{N-1-j}q^j$.

\begin{theorem}[bin probabilities]\label{thm:pmf}
Let $N\ge1$. Then $P_N(0)=P_N(N)=\tfrac12p^{N-1}$, and for $1\le k\le N-1$
\[
P_N(k)=\frac12\sum_{j=1}^{N-1}W(N,k,j)\,p^{N-1-j}q^{\,j},
\]
where $W(N,k,j)$ counts the words with $k$ letters $\sR$, $N-k$ letters $\sL$
and exactly $j$ switches. Writing $a=k$ and $b=N-k$, for every integer $i\ge1$
\[
W(N,k,2i-1)=2\binom{a-1}{i-1}\binom{b-1}{i-1},\qquad
W(N,k,2i)=\binom{a-1}{i}\binom{b-1}{i-1}+\binom{a-1}{i-1}\binom{b-1}{i}
\]
with $\binom nr=0$ for integers $n\ge0$ and $r$ outside $0\le r\le n$.
\end{theorem}

\begin{proof}
By \eqref{eq:wordprob} it suffices to count the words in bin $k$ with $j$
switches. We count them by their run lengths. A word with $j$ switches
consists of $j+1$ maximal runs, alternately of letters $\sR$ and of letters
$\sL$. If $k\in\{0,N\}$ there is one word and it has no switch. Let $a,b\ge1$.
If $j+1=2i$ there are $i$ runs of each letter. The run lengths form a
composition of $a$ into $i$ parts and a composition of $b$ into $i$ parts, and
there are $\binom{a-1}{i-1}\binom{b-1}{i-1}$ such pairs. The two choices of the
first letter give the factor $2$. If
$j+1=2i+1$ the word starts and ends with the same letter. Starting with $\sR$
gives $i+1$ runs of $\sR$ and $i$ runs of $\sL$, so
$\binom{a-1}{i}\binom{b-1}{i-1}$ words, and starting with $\sL$ gives the
other summand.
\end{proof}

For instance $P_2=(\tfrac p2,\,q,\,\tfrac p2)$, $P_3(1)=\tfrac12(1-p^2)$ and
$P_4(2)=q(1-p+p^2)$. Each interior bin needs at least one switch, while the
endpoints need none. So the endpoints win when $p$ is close to $1$ (the walk
rarely turns), and
Section~\ref{sec:middle} asks how close.

\begin{proposition}[recurrence and generating function]\label{prop:gf}
Let $P_N^{\sR}(k)$ and $P_N^{\sL}(k)$ be the probabilities of being in bin $k$
after $N$ steps with last letter $\sR$, respectively $\sL$, and set them to $0$
outside $0\le k\le N$. Then $P_1^{\sR}(1)=P_1^{\sL}(0)=\tfrac12$,
$P_1^{\sR}(0)=P_1^{\sL}(1)=0$, and for $N\ge2$
\begin{equation}\label{eq:dp}
P_N^{\sR}(k)=p\,P_{N-1}^{\sR}(k-1)+q\,P_{N-1}^{\sL}(k-1),\qquad
P_N^{\sL}(k)=p\,P_{N-1}^{\sL}(k)+q\,P_{N-1}^{\sR}(k).
\end{equation}
Consequently, in formal variables $x$ and $t$,
\begin{equation}\label{eq:gf}
\sum_{N\ge0}\sum_{k=0}^{N}P_N(k)\,x^kt^N=\frac{1-(p-\tfrac12)(1+x)\,t}{1-p(1+x)\,t+(2p-1)\,x\,t^2},
\end{equation}
and for all $N\ge2$ and all $k$, with $P_M(k)=0$ outside $0\le k\le M$,
\begin{equation}\label{eq:rec}
P_N(k)=p\bigl(P_{N-1}(k)+P_{N-1}(k-1)\bigr)-(2p-1)\,P_{N-2}(k-1).
\end{equation}
\end{proposition}

\begin{proof}
The recursion \eqref{eq:dp} follows by conditioning on $w_{N-1}$, since $w_N$
equals $w_{N-1}$ with probability $p$.
Let $H^{\sR}=\sum_{N,k}P_N^{\sR}(k)x^kt^N$ and define $H^{\sL}$ alike. By
\eqref{eq:dp}, $H^{\sR}=xt(\tfrac12+pH^{\sR}+qH^{\sL})$ and
$H^{\sL}=t(\tfrac12+pH^{\sL}+qH^{\sR})$. Solving these two linear equations
and adding $1$ for $N=0$ gives \eqref{eq:gf}, which is Gillis's generating
function~\cite[Eq.~(3.1)]{gillis1955} after changing from position to bin number. Multiplying
\eqref{eq:gf} by its denominator and comparing the coefficients of $x^kt^N$
with $N\ge2$ gives \eqref{eq:rec}.
\end{proof}

\begin{remark}[integer numerators]\label{rem:tables}
For $p=\alpha/(\alpha+\beta)$ with integers $\alpha,\beta\ge0$, not both $0$,
put $n_N(k)=2(\alpha+\beta)^{N-1}P_N(k)$ for $N\ge1$. By
\eqref{eq:wordprob}, $n_N(k)$ is the sum of $\alpha^{\rep(w)}\beta^{\chg(w)}$
over the words with $k$ letters $\sR$ and $N-k$ letters $\sL$, which shows that
it is an integer.
\end{remark}

\subsection{Variance and the normal limit}\label{sec:clt}

The bin $K_N$ counts the visits of the chain $w_1,\dots,w_N$ to the letter
$\sR$. This chain leaves each letter with probability $q$ and starts from its
stationary law $(\frac12,\frac12)$. So $K_N$ has the Markov binomial law
$\mathrm{Bin}(N,q,q,(\frac12,\frac12))$ of Dekking and
Kong~\cite{dekkingkong2011}. The transition matrix of the chain has second
eigenvalue $1-2q=2p-1$.

\begin{proposition}[mean and variance]\label{prop:var}
Let $0<p<1$ and $\lambda=2p-1$. Then $\E K_N=N/2$ and
\[
\Var K_N=\frac14\Bigl(N\,\frac{1+\lambda}{1-\lambda}-\frac{2\lambda\,(1-\lambda^N)}{(1-\lambda)^2}\Bigr)
=\frac{Np}{4(1-p)}-\frac{(2p-1)\bigl(1-(2p-1)^N\bigr)}{8(1-p)^2}.
\]
\end{proposition}

\begin{proof}
This is \cite[Eq.~(2.3) and Proposition~2.1]{dekkingkong2011} with both of
their switching probabilities equal to $q$
and the stationary start, for which both excentricities vanish and both
stationary weights are $\frac12$. The second form follows from
$(1+\lambda)/(1-\lambda)=p/q$ and $1-\lambda=2q$.
\end{proof}

\begin{theorem}[normal limit]\label{thm:clt}
Let $0<p<1$ be fixed. Then $(K_N-N/2)/\sqrt N$ converges in distribution, as
$N\to\infty$, to the normal law with mean $0$ and variance $p/(4(1-p))$.
\end{theorem}

\begin{proof}
By \eqref{eq:dp}, the vector $v_N(\theta)$ with entries
$\E[e^{i\theta X_N};\sigma_N=\pm1]$ satisfies $v_{N+1}=T_\theta v_N$, where $T_\theta$ has rows
$(pe^{i\theta},qe^{i\theta})$ and $(qe^{-i\theta},pe^{-i\theta})$. The
eigenvalues of $T_\theta$ at $\theta=0$ are $1$ and $2p-1$, so for small $|\theta|$ they
stay simple, with
$\lambda_+(\theta)=p\cos\theta+\sqrt{q^2-p^2\sin^2\theta}=1-\frac p{2q}\theta^2+O(\theta^4)$
and $\lambda_-$ near $2p-1$. Since $v_1(0)$ is an eigenvector for the
eigenvalue $1$, the component of $v_1(\theta)$ along the eigenvector of
$\lambda_+(\theta)$ tends to $1$ as $\theta\to0$, the term with $\lambda_-^N$ tends
to $0$ since $|\lambda_-|<1$, and at
$\theta=s/\sqrt N$ we get $\E e^{i\theta X_N}\to\exp(-\frac p{2q}s^2)$. L\'evy's
continuity theorem~\cite[Theorem~26.3]{billingsley1995} gives the normal
limit of $X_N/\sqrt N$ with variance $p/q$, and $K_N-N/2=X_N/2$.
\end{proof}

Compared with the simple random walk, persistence multiplies the width of the
bell by $\sqrt{p/q}$, so it widens the bell for $p>\frac12$ and narrows it for
$p<\frac12$.

\section{The central crossing at every \texorpdfstring{$N$}{N}}\label{sec:middle}

In this section the chain has free ends. Let $N\ge2$, let $m_N=\lfloor N/2\rfloor$
and let $p>0$. By Theorem~\ref{thm:pmf} the ratio of the middle bin to an
endpoint is
\begin{equation}\label{eq:SN}
\frac{P_N(m_N)}{P_N(0)}=S_N(u)=\sum_{j=1}^{N-1}W(N,m_N,j)\,u^j ,
\end{equation}
a polynomial in $u$ with nonnegative integer coefficients and
$W(N,m_N,1)=2$. The middle and the endpoints are level exactly when
$S_N(u)=1$.

We first explain heuristically why the two are level when $\tau$ is near $\log N$.
Let $a=m_N$ and
$b=N-m_N$, both about $N/2$. A word in the middle bin with $2r+1$ switches has
$r+1$ runs of each letter, and there are
$2\binom{a-1}r\binom{b-1}r\approx2(N/2)^{2r}/(r!)^2$ of them when $r$ is small
compared with $\sqrt N$. So these words contribute about
$2u\sum_r(Nu/2)^{2r}/(r!)^2=2uI_0(Nu)$ to $S_N(u)$, and in the same way the
words with an even number of switches contribute about $2uI_1(Nu)$. For large
$w=Nu$ we have $I_0(w)+I_1(w)\approx2e^w/\sqrt{2\pi w}$, so
$S_N(u)\approx4we^w/(N\sqrt{2\pi w})$. Setting this equal to $1$ gives
$w+\frac12\log w\approx L+\frac12\log(\pi/8)$. So at the crossing $w$, and with
it $\tau$, is close to $L-\frac12\log L+c_0$ with $c_0=\frac12\log(\pi/8)$.

\begin{theorem}\label{thm:root}
\begin{enumerate}[(a)]
\item For every $N\ge2$ there is exactly one $p=p_N\in[0,1]$ with
$P_N(m_N)=P_N(0)$, and $p_N\ge\tfrac23$. The middle is more likely than the
endpoints for $p<p_N$ and less likely for $p>p_N$.
\item The first values are $p_2=\tfrac23$, $p_3=1/\sqrt2$ and the real root
$p_4$ of $3p^3-4p^2+4p-2$, and $p_N<p_{N+1}$ for all $N\ge2$.
\item As $N\to\infty$ we have $p_N\to1$, and more precisely
$1-p_N\sim(\log N)/N$.
\end{enumerate}
\end{theorem}

\begin{proof}
(a) At $p=0$ the endpoints have probability $0$ and the middle does not. For
$p>0$ the equation is $S_N(u)=1$ by \eqref{eq:SN}. The coefficients of $S_N$
are nonnegative and the linear one is $2$, so $S_N$ is strictly increasing on
$[0,\infty)$ with $S_N(0)=0$ and $S_N(u)\ge2u$. By the intermediate value
theorem there is exactly one root $u_N\in(0,\tfrac12]$. The map
$u\mapsto p=1/(1+u)$ is a decreasing bijection from $[0,\infty)$ to $(0,1]$.
Hence $p_N=1/(1+u_N)\ge\tfrac23$, and the middle wins exactly when $u>u_N$.

(b) By Theorem~\ref{thm:pmf}, $S_2=2u$, $S_3=2u+u^2$ and $S_4=2u+2u^2+2u^3$,
which give the three values. For $p_4$,
substitute $u=(1-p)/p$ in $S_4(u)=1$. For the monotonicity we compare
coefficients. Going from $N$ to $N+1$, one entry of the middle pair
$(a,b)=(\lfloor N/2\rfloor,\lceil N/2\rceil)$ grows by one and the other is
unchanged. The formulas of Theorem~\ref{thm:pmf} are symmetric in $(a,b)$, and
every $\binom{a-1}{\cdot}$ is nondecreasing in $a$. So
$W(N+1,m_{N+1},j)\ge W(N,m_N,j)$ for every $j$. The coefficient of $u^2$ is
$a+b-2=N-2$, which increases strictly. Hence $S_{N+1}(u)>S_N(u)$ for $u>0$, so
$u_{N+1}<u_N$.

(c) We prove (c) inside the proof of Theorem~\ref{thm:bessel}, using only
part~(i) of that theorem.
\end{proof}

Figure~\ref{fig:row10} shows the law of a chain of nine spins below, at and
above its crossing.

\begin{figure}[t]
\centering
\includegraphics{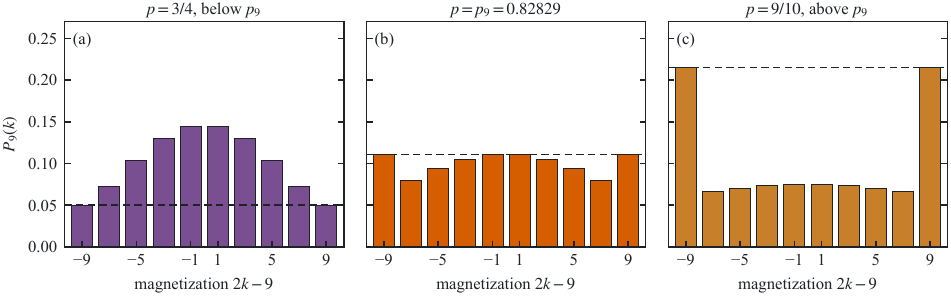}
\caption{The bin probabilities $P_9(k)$ of a chain of nine spins against the
magnetization $2k-9$, at $p=\frac34$ (a), at the crossing
$p_9=0.82829\ldots$ (b) and at $p=\frac9{10}$ (c), on a common vertical scale. The dashed line is the
endpoint height. At $p=p_9$ the two central bins, with magnetization $-1$ and
$1$, are exactly as likely as each endpoint.}
\label{fig:row10}
\end{figure}

\subsection{A Bessel bound at every \texorpdfstring{$N$}{N}}

The heuristic above replaced the number $\binom{a-1}{r}$ of ways to split $a$
letters into $r+1$ runs by $(N/2)^r/r!$. After this change the run count is
exactly the power series of $I_0$ and $I_1$. The next lemma bounds the error
this makes.

\begin{lemma}[clipped products]\label{lem:clipped}
(a) If $t_1,\dots,t_r\ge0$ and $(v)_+=\max(v,0)$, then
$0\le1-\prod_{k=1}^r(1-t_k)_+\le\sum_{k=1}^rt_k$.
(b) For integers $m\ge1$ and $r\ge0$,
$U_r(m):=\frac{r!}{m^r}\binom{m-1}{r}=\prod_{k=1}^r\bigl(1-\frac km\bigr)_+$,
so $0\le1-U_r(m)\le\frac{r(r+1)}{2m}$.
(c) For $x\ge0$,
\begin{gather*}
\sum_{r\ge0}\frac{r(r+1)x^{2r}}{(r!)^2}=x^2I_0(2x)+xI_1(2x),\\
\sum_{r\ge0}\frac{r(r+1)x^{2r+1}}{r!(r+1)!}=x^2I_1(2x),\qquad
\sum_{r\ge0}\frac{(r+1)x^{2r+1}}{r!(r+1)!}=xI_0(2x).
\end{gather*}
\end{lemma}

\begin{proof}
(a) Every factor lies in $[0,1]$, so the product is at most $1$. If some
$t_k\ge1$ the product is $0$ and the sum is at least $1$. Otherwise
$\prod_k(1-t_k)\ge1-\sum_kt_k$ by induction on $r$ (the Weierstrass
product inequality). (b) Note that
$\frac{r!}{m^r}\binom{m-1}{r}=\prod_{k=1}^r\frac{m-k}m$. For $r<m$ every factor
is positive, and for $r\ge m$ both sides vanish, the right side through the
factor $k=m$. The bound is (a) with $t_k=k/m$. (c) Use the series of $I_0(2x)$
and $I_1(2x)$ from Section~\ref{sec:setting}. Write $r(r+1)=r^2+r$ in the
first sum, cancel the factors $r$ and $r+1$ against the factorials, and shift
the index.
\end{proof}

\begin{theorem}[uniform Bessel approximation]\label{thm:bessel}
Put $h=N/2$ and $F(x)=2x\bigl(I_0(2x)+I_1(2x)\bigr)$, and recall that
$u_N=(1-p_N)/p_N$.
\begin{enumerate}[(i)]
\item For every $N\ge2$ and $x>0$,
\begin{equation}\label{eq:uniformbessel}
0\le1-\frac{hS_N(x/h)}{F(x)}\le\frac{x(x+1)}h.
\end{equation}
\item Let $x_N=hu_N$, let $x_B$ be the unique positive root of $F(x_B)=h$, and
put $\delta_N=x_N(x_N+1)/h$. Then $x_B\le x_N$, and
$x_N-x_B\le-\frac12\log(1-\delta_N)$ whenever $\delta_N<1$.
\item With $u_B=x_B/h$ we have $0\le u_N-u_B=O(L^2/N^2)$ as $N\to\infty$.
\end{enumerate}
\end{theorem}

Section~\ref{sec:every-n-sub} uses part~(i) to locate $p_N$ at every $N$.

\begin{proof}
(i) \emph{Step 1. The clipped factors.} Let $a=\lfloor N/2\rfloor$ and
$b=\lceil N/2\rceil$. As in Lemma~\ref{lem:clipped}(b),
$A_r=\frac{r!}{h^r}\binom{a-1}{r}=\prod_{k=1}^r(1-\frac{h-a+k}h)_+$, and $B_r$
is the same with $b$ in place of $a$. Since $a+b=2h$ and $h-a\in\{0,\frac12\}$,
the $2r$ numbers $(h-a+k)/h$ and $(h-b+k)/h$ are nonnegative with sum
$r(r+1)/h$, so Lemma~\ref{lem:clipped}(a) gives $0\le1-A_rB_r\le r(r+1)/h$. In
the same way the deficits $1-A_{r+1}B_r$ and $1-A_rB_{r+1}$ are at most
$((r+1)^2+h-a)/h$ and $((r+1)^2+h-b)/h$. Averaging cancels the parity offsets
(the shifts $h-a$ and $h-b$ add to $0$), so $0\le1-\frac12(A_{r+1}B_r+A_rB_{r+1})\le(r+1)^2/h$.

\emph{Step 2. Summing.} With $u=x/h$ the run-count formula of
Theorem~\ref{thm:pmf} becomes
\[
S_N(u)=2u\Bigl[\sum_{r\ge0}\frac{x^{2r}}{(r!)^2}A_rB_r
+\sum_{r\ge0}\frac{x^{2r+1}}{r!(r+1)!}\,\frac{A_{r+1}B_r+A_rB_{r+1}}2\Bigr].
\]
Without the factors $A_r,B_r$ the bracket is $I_0(2x)+I_1(2x)$ and the right
side is $F(x)/h$. By Step~1, Lemma~\ref{lem:clipped}(c) and $(r+1)^2=r(r+1)+(r+1)$, $F(x)/h-S_N(u)$
lies between $0$ and $2u/h$ times
$\sum_rr(r+1)\frac{x^{2r}}{(r!)^2}+\sum_r(r+1)^2\frac{x^{2r+1}}{r!(r+1)!}=x(x+1)\bigl(I_0(2x)+I_1(2x)\bigr)$.
Dividing by $F(x)/h=2u(I_0(2x)+I_1(2x))$ proves (i). This covers every $r$,
including those beyond the degree of the polynomial, where the clipped products
vanish and Lemma~\ref{lem:clipped}(a) still applies.

(ii) Differentiating the series term by term gives $F'=2F+2I_0(2x)$. So $F$
increases from $0$ to $\infty$, $x_B$ is unique, and $(\log F)'\ge2$. At
$x=x_N$ we have $hS_N(u_N)=h$, so (i) gives $h\le F(x_N)$, hence $x_B\le x_N$,
and $F(x_N)\le h/(1-\delta_N)$ when $\delta_N<1$. Integrating $(\log F)'\ge2$
from $x_B$ to $x_N$ gives $2(x_N-x_B)\le\log(F(x_N)/h)\le-\log(1-\delta_N)$.

\emph{Proof of Theorem~\ref{thm:root}(c).} We use only part~(i). Note that
$x^{2r}/(r!)^2\le(2x)^{2r}/(2r)!$ and
$x^{2r+1}/(r!(r+1)!)\le(2x)^{2r+1}/(2r+1)!$ for $x\ge0$, since
$\binom{2r}r\le4^r$ and $\binom{2r+1}r\le2^{2r+1}$. So
$I_0(2x)+I_1(2x)\le e^{2x}$ and $F(x)\le2xe^{2x}$. Fix $\theta\in(0,1)$ and
put $x_\pm=(1\pm\theta)L/2$. Then $F(x_-)\le(1-\theta)LN^{1-\theta}<h$ for
large $N$, so $hS_N(x_-/h)<h$ by (i), and $x_-<x_N$ since $S_N$ increases. On
the other side $F(x_+)\ge2x_+I_0(2x_+)$, which is at least a constant times
$\sqrt L\,N^{1+\theta}$ for large $N$ by the large-argument form of $I_0$,
while $x_+(x_++1)/h\to0$. So $hS_N(x_+/h)\ge F(x_+)(1-x_+(x_++1)/h)>h$ for
large $N$, and $x_N<x_+$. Hence $Nu_N=2x_N\sim L$, so $u_N\to0$ and
$z_N=Nu_N/(1+u_N)\sim L$, which is (c).

(iii) By Theorem~\ref{thm:root}(c), $x_N\sim L/2$. So $\delta_N=O(L^2/N)$,
and (ii) gives $x_N-x_B=O(L^2/N)$. Divide by $h$.
\end{proof}

\subsection{The crossing at every \texorpdfstring{$N$}{N}}\label{sec:every-n-sub}

We measure the crossing by $z_N=N(1-p_N)=Nu_N/(1+u_N)$, the value of $\tau$ at
$p=p_N$. Let $\psi(z)=z+\frac12\log z$ and $c_0=\frac12\log(\pi/8)=-0.46735\ldots$,
and let $\zeta_N$ be the positive root of
\begin{equation}\label{eq:zetaN}
\psi(\zeta_N)=L+c_0,\quad\text{that is,}\quad 8\zeta_Ne^{2\zeta_N}=\pi N^2
\quad\text{or}\quad\zeta_N=\tfrac12W_0(\pi N^2/4),
\end{equation}
with $W_0$ the principal branch of the Lambert function. This is the
heuristic equation from the start of the section with $w$ replaced by $\zeta$.
Since $\zeta e^{2\zeta}$ increases, $\zeta_N$ increases with $N$. Also
$\zeta_N\le L$, since $\psi(L)\ge L+c_0$ when $L\ge\pi/8$. Finally let
$\varepsilon_N=\zeta_N(\zeta_N+1)/N+1/(16\zeta_N^2)$.

\begin{theorem}[the crossing at every $N$]\label{thm:every-n}
For every $N\ge2$,
\[
\zeta_N+\frac1{8\zeta_N}-\varepsilon_N<z_N<\zeta_N+\frac1{8\zeta_N}.
\]
\end{theorem}

The width $\varepsilon_N$ tends to $0$, since $\zeta_N\le L$ and
$\zeta_N\to\infty$. It is $0.440$, $0.163$, $0.0387$ and $0.00778$ at
$N=10$, $10^2$, $10^3$ and $10^4$ (Table~\ref{tab:every-n}).

The two terms $\zeta_N$ and $1/(8\zeta_N)$ come from the following heuristic.
Let
$\Phi(w)=F(w/2)=w\bigl(I_0(w)+I_1(w)\bigr)$ and
$\cE(z)=\sqrt{\pi z/2}\,e^{-z}\bigl(I_0(z)+I_1(z)\bigr)$, which tends to $1$ as
$z\to\infty$. Put $w_*=Nu_N$ and $\Lambda_N=2\Phi(w_*)/N$. At the crossing, the identity
\eqref{eq:bracket-log} of Lemma~\ref{lem:bracket} and the bounds
\eqref{eq:uniformbessel} give
$\psi(w_*)+\log \cE(w_*)=L+c_0+\log\Lambda_N$ with
$0\le\log\Lambda_N\le-\log(1-w_*(w_*+2)/(2N))$ whenever $w_*(w_*+2)<2N$. With $\cE=\Lambda_N=1$ this
would say $w_*=\zeta_N$. Since $\log \cE(w)=-1/(8w)+O(w^{-2})$ and
$z_N=w_*+O(w_*^2/N)$, this equation suggests
$z_N=\zeta_N+1/(8\zeta_N)+O(\zeta_N^{-2}+\zeta_N^2/N)$. The theorem makes this
error explicit and one-sided. Its proof uses two lemmas. Lemma~\ref{lem:bracket}
turns a comparison of $z_N$ with a test value into a sign condition, and
Lemma~\ref{lem:bessel-bounds} bounds $\cE$ with explicit constants.

\begin{lemma}[root bracket]\label{lem:bracket}
Let $N\ge2$ and $w>0$. Then
\begin{equation}\label{eq:bracket-log}
\log\Phi(w)-\log\tfrac N2=\psi(w)-\psi(\zeta_N)+\log \cE(w).
\end{equation}
(a) If $\Phi(w)\le N/2$, then $Nw/(N+w)\le z_N$, strictly if $\Phi(w)<N/2$.
(b) If $\Phi(w)\bigl(1-\frac{w(w+2)}{2N}\bigr)>N/2$, then $z_N<Nw/(N+w)$.
\end{lemma}

\begin{proof}
Note that $\log\Phi(w)=\psi(w)+\frac12\log\frac2\pi+\log \cE(w)$ and
$\log\frac N2-\psi(\zeta_N)=-\log2-c_0=\frac12\log\frac2\pi$, which gives
\eqref{eq:bracket-log}. Now apply \eqref{eq:uniformbessel} with $x=w/2$. Then
$F(x)=\Phi(w)$ and $x(x+1)/h=w(w+2)/(2N)$, so
\[
\Phi(w)\Bigl(1-\frac{w(w+2)}{2N}\Bigr)\le\frac N2\,S_N\Bigl(\frac wN\Bigr)\le\Phi(w).
\]
(a) The hypothesis gives $S_N(w/N)\le1=S_N(u_N)$, strictly if $\Phi(w)<N/2$.
Since $S_N$ is strictly increasing, $w/N\le u_N$, strictly in that case. The
map $v\mapsto Nv/(1+v)$ increases and sends $w/N$ to $Nw/(N+w)$ and $u_N$ to
$z_N$. (b) Now $S_N(w/N)>1$, so $w/N>u_N$.
\end{proof}

\begin{lemma}[Bessel bounds]\label{lem:bessel-bounds}
For every $z>0$,
\[
1-\frac1{8z}-\frac3{128z^2}-\frac{45}{512z^3}\le \cE(z)\le1-\frac1{8z}-\frac3{128z^2}+\frac{15}{512z^3},
\]
and $\cE(z)\le1$. In particular $\cE(z)\le1-\frac1{8z}$ for $z\ge\frac54$.
\end{lemma}

The lemma is the large-argument expansion of $I_0+I_1$ with explicit errors.
We prove it in Appendix~\ref{app:bessel-bounds} by bounding $\sqrt{1-y}$ by cubics
in an integral form of $\cE$.

\begin{proof}[Proof of Theorem~\ref{thm:every-n}]
\emph{Step 1. The split at $N=175$.} By \eqref{eq:zetaN},
$N=\sqrt{8\zeta_N/\pi}\,e^{\zeta_N}$, and the right side increases in
$\zeta_N$. Since $\sqrt{32/\pi}\,e^4=174.25\ldots$, we have $\zeta_N\ge4$
exactly when $N\ge175$. Write $\zeta=\zeta_N$ and $Y=\zeta+\frac1{8\zeta}$.

\emph{Step 2. The upper bound for $N\ge175$.} We apply
Lemma~\ref{lem:bracket}(b) at $\bar w=NY/(N-Y)$, for which $N\bar w/(N+\bar w)=Y$. By
\eqref{eq:bracket-log} it suffices that
$\mathcal G=\psi(\bar w)-\psi(\zeta)+\log \cE(\bar w)+\log(1-\bar w(\bar w+2)/(2N))>0$. Expanding each
term with Lemma~\ref{lem:bessel-bounds} gives
$\mathcal G\ge\frac1{32\zeta^2}+\frac{\zeta(\zeta-1)}{2N}+O(\zeta^{-3}+N^{-1}+\zeta^4N^{-2})$,
and Appendix~\ref{app:every-n} makes every term explicit for $\zeta\ge4$.

\emph{Step 3. The lower bound for $N\ge175$.} We apply
Lemma~\ref{lem:bracket}(a) at $\bar w'=NY'/(N-Y')$ with
$Y'=Y-\varepsilon_N$. It suffices that
$\mathcal G'=\psi(\bar w')-\psi(\zeta)+\log \cE(\bar w')<0$, and we get
$\mathcal G'\le-\frac3{128\zeta^2}-\frac\zeta N+O(\zeta^{-3}+N^{-1}+\zeta^4N^{-2})$,
again with explicit constants for every $\zeta\ge4$ in
Appendix~\ref{app:every-n}. The term $\zeta(\zeta+1)/N$ in $\varepsilon_N$
absorbs the contributions of order $\zeta^2/N$, so that only $-\zeta/N$ remains.

\emph{Step 4. Short chains, by certificates.} Here $2\le N\le174$. Let
$d=10^9$, let $c=\lfloor du_N\rfloor$ and let $c'=c+1$. At $N=2$, where
$du_2=5\cdot10^8$ is an integer, we take $c=du_2-1$ and $c'=c+2$. The sign of $S_N(c/d)-1$ is the sign of $n_N(m_N)-d^{N-1}$, where
$n_N(m_N)=d^{N-1}S_N(c/d)$ is the integer of Remark~\ref{rem:tables} at
$(\alpha,\beta)=(d,c)$ (Appendix~\ref{app:certificates}). We computed these
integers at $c$ and $c'$ by two exact methods, the run-count sum of
Theorem~\ref{thm:pmf} and the recursion \eqref{eq:dp} with integer weights
$d,c$ in place of $p,q$. They agreed for every $N\le174$ and gave
$c/d<u_N<c'/d$, so $z(c/d)<z_N<z(c'/d)$ with $z(u)=Nu/(1+u)$. We then checked
$\zeta_N+\frac1{8\zeta_N}-\varepsilon_N<z(c/d)$ and
$z(c'/d)<\zeta_N+\frac1{8\zeta_N}$ with outward interval arithmetic
at 80 decimal digits. To enclose $\pi$, we used
$\pi=16\arctan(1/5)-4\arctan(1/239)$ and 80 terms of each alternating
series, with the first omitted term bounding its remainder. Each arithmetic
operation rounds outward. For a logarithm, we take the neighboring decimal
on each side of the correctly rounded endpoint value.

Newton's method only proposes a value of $\zeta_N$. We certify the bracket
of half-width $10^{-45}$ around it by enclosing
$\psi(\zeta)-L-c_0$ at both ends and checking opposite signs.
Evaluating both comparisons on this whole bracket gives margins greater than
$0.03319$ (upper) and $0.08559$ (lower), with the smallest at $N=174$.
The script \nolinkurl{verify_certified_bounds.py} replays these checks from
the stored rational endpoints, including all rounding and remainder bounds.
Lean checks the short-chain comparisons independently with rational bounds on $\zeta_N$ and $\pi$. We
also ran the certificates up to $N=1000$ as a sanity check (both methods up to
$400$, then only the run-count sum), and they all passed.
\end{proof}

\begin{table}[ht]
\centering\small
\begin{tabular}{@{}rllrrr@{}}
\toprule
$N$ & $z_N$ & $\zeta_N$ & $z_N-\zeta_N-\frac1{8\zeta_N}$ & $-\varepsilon_N$ & $R_N$\\
\midrule
2 & 0.666666666667 & 0.5368290974 & $-0.1030112$ & $-0.62938$ & $+0.050643$\\
10 & 1.615958871376 & 1.6001732978 & $-0.0623310$ & $-0.44048$ & $-0.056660$\\
100 & 3.503367368092 & 3.5100055489 & $-0.0422507$ & $-0.16337$ & $-0.043265$\\
174 & 3.996779704730 & 3.9987131703 & $-0.0331935$ & $-0.11878$ & $-0.033450$\\
175 & 4.001967989170 & 4.0038072816 & $-0.0330596$ & $-0.11838$ & $-0.033304$\\
1000 & 5.590351166975 & 5.5807388611 & $-0.0127862$ & $-0.03873$ & $-0.011887$\\
$10^4$ & 7.734265514999 & 7.7210118331 & $-0.0029359$ & $-0.00778$ & $-0.002051$\\
\bottomrule
\end{tabular}
\caption{The crossing against Theorem~\ref{thm:every-n}, which says that the fourth
column lies strictly between the fifth and $0$, and the remainder $R_N$ of
\eqref{eq:crossing-eq}, which lies between $-(z_N^2+z_N)/N$ and $1/(16z_N^2)$. We computed $z_N$ by run-count bisection in 50-digit decimal arithmetic
and by a floating-point dynamic program, which agree to the digits shown.}
\label{tab:every-n}
\end{table}

The bracket is in terms of $\zeta_N$ (a value of the Lambert function). The
next corollary states it for $p_N$ and
expands it in $L=\log N$.

\begin{corollary}\label{cor:every-n}
Put $\nu_N=L-\frac12\log L+c_0$, $Y_N=\zeta_N+\frac1{8\zeta_N}$ and
$Y_N'=Y_N-\varepsilon_N$.
\begin{enumerate}[(a)]
\item For every $N\ge2$ we have $0<Y_N'<Y_N<N$,
$1-Y_N/N<p_N<1-Y_N'/N$ and, for the coupling
$J_N^{\mathrm F}=-\frac12\log u_N$ at the crossing,
\[
\frac12\log\Bigl(\frac N{Y_N}-1\Bigr)<J_N^{\mathrm F}<\frac12\log\Bigl(\frac N{Y_N'}-1\Bigr).
\]
\item For every $N\ge2$, $\nu_N\le\zeta_N\le\nu_N+\frac{\log L-2c_0}{4\nu_N}$ and
\begin{equation}\label{eq:secondorder}
-\varepsilon_N<z_N-\nu_N<\frac{\log L-2c_0+\frac12}{4\nu_N}.
\end{equation}
In particular $z_N=L-\frac12\log L+c_0+O(\log L/L)$ as $N\to\infty$.
\item As $N\to\infty$,
\begin{equation}\label{eq:thirdorder}
z_N=\nu_N+\frac{\frac14\log L-\frac12c_0+\frac18}{L}+O\Bigl(\frac{(\log L)^2}{L^2}\Bigr).
\end{equation}
\item (Computer-assisted.) For every $N\ge2774$ we have $z_N>\zeta_N$.
\end{enumerate}
\end{corollary}

\begin{proof}
(a) Since $\frac12<\zeta_2\le\zeta_N$, we have $\frac1{16\zeta_N^2}<\frac1{8\zeta_N}$, and
since $\zeta_N\le L<N-1$, we have $\zeta_N(\zeta_N+1)/N<\zeta_N$. So
$\varepsilon_N<Y_N$, that is, $Y_N'>0$. Also $Y_N<L+\frac14<N$. The bounds are
Theorem~\ref{thm:every-n} with $p_N=1-z_N/N$ and, by \eqref{eq:zN},
$J_N^{\mathrm F}=\frac12\log(N/z_N-1)$, which decreases in $z_N$.

(b) By \eqref{eq:zetaN}, $\zeta_N=L+c_0-\frac12\log\zeta_N$. Since
$\zeta_N\le L$, this gives $\zeta_N\ge\nu_N$, and $\nu_N>0$ since it increases
in $N$ and $\nu_2>0.40$. Also $\log y\le y-1$ gives
$\zeta_N-\nu_N=\frac12\log(L/\zeta_N)\le(L-\zeta_N)/(2\zeta_N)$, and
$L-\zeta_N=\frac12\log\zeta_N-c_0\le\frac12\log L-c_0$. Hence
$\zeta_N-\nu_N\le(\log L-2c_0)/(4\nu_N)$. For $z_N$, add
Theorem~\ref{thm:every-n} and $0<\frac1{8\zeta_N}\le\frac1{8\nu_N}$. The last
claim follows since $\varepsilon_N=O(L^2/N+L^{-2})$.

(c) Put $\mu=\zeta_N-\nu_N=\frac12\log(L/\zeta_N)$. By (b),
$L-\zeta_N=\frac12\log L-c_0+O(\log L/L)$, so
$\mu=\frac{L-\zeta_N}{2L}+O\bigl(\frac{(\log L)^2}{L^2}\bigr)
=\frac{\frac14\log L-\frac12c_0}L+O\bigl(\frac{(\log L)^2}{L^2}\bigr)$.
Also $\frac1{8\zeta_N}=\frac1{8L}+O(\frac{\log L}{L^2})$ and
$\varepsilon_N=O(L^2/N+L^{-2})$, and Theorem~\ref{thm:every-n} gives (c).

(d) By Theorem~\ref{thm:every-n}, $z_N-\zeta_N>\frac1{8\zeta_N}-\varepsilon_N$,
which is positive when $f(\zeta_N)>1$, where
$f(\zeta)=n(\zeta)\bigl(\frac1{8\zeta}-\frac1{16\zeta^2}\bigr)/(\zeta(\zeta+1))$
and $n(\zeta)=\sqrt{8\zeta/\pi}\,e^\zeta$, so that $n(\zeta_N)=N$. The
logarithmic derivative of $f$ is
$1+\frac1{2\zeta}-\frac2\zeta+\frac1{2\zeta^2-\zeta}-\frac1{\zeta+1}>1-\frac3{2\zeta}-\frac1{\zeta+1}>0$
for $\zeta\ge3$. Since $\zeta_N$ increases with $N$, once $f(\zeta_N)>1$ it
stays above $1$. This step is computer-assisted. With the outward intervals
of Step~4 of the proof of Theorem~\ref{thm:every-n}, we evaluate
$f(\zeta_{2774})$ with $n(\zeta_{2774})$ replaced by the exact integer $2774$,
and obtain $f(\zeta_{2774})-1>2.0\times10^{-4}$ on the whole bracket. Since $\zeta_{2774}>3$, this proves (d). The same interval
check gives $f(\zeta_{2773})<1$.
The conclusion for every $N\ge2774$ is also proved in Lean. That proof uses
rational Taylor bounds for $\zeta_{2774}$ and an exponential comparison
to extend the bound to larger $N$.
\end{proof}

The heuristic before Lemma~\ref{lem:bracket} suggests that $\psi(z_N)\approx L+c_0+1/(8z_N)$. The next
proposition proves this at every $N$ with an explicit remainder.

\begin{proposition}[the crossing equation]\label{prop:crossing-eq}
For every $N\ge2$,
\begin{equation}\label{eq:crossing-eq}
z_N+\tfrac12\log z_N=L+c_0+\frac1{8z_N}+R_N,\qquad
-\frac{z_N^2+z_N}N\le R_N\le\frac1{16z_N^2}.
\end{equation}
We also have $L/2\le z_N\le L$, so $|R_N|\le\frac1{4L^2}+\frac{L^2+L}N$ for every
$N\ge2$.
\end{proposition}

We prove Proposition~\ref{prop:crossing-eq} in Appendix~\ref{app:crossing-eq}.
The proof writes $R_N$ exactly, using \eqref{eq:uniformbessel} and
\eqref{eq:bracket-log} at $w=w_*$ (formula \eqref{eq:RN-formula}). For
$N\ge175$, Lemma~\ref{lem:bessel-bounds} puts the Bessel part of this formula
between $0$ and $0.0559/z_N^2$, and the other terms lie between
$-(z_N^2+z_N)/N$ and $0$. For $N\le174$ the proof uses the same certificates as
Theorem~\ref{thm:every-n}.

\begin{remark}[the sign of $z_N-\zeta_N$, computer-assisted]\label{rem:sign}
We have $z_N>\zeta_N$ for $N\le15$, $z_N<\zeta_N$ for $16\le N\le218$, and
$z_N>\zeta_N$ for every $N\ge219$. For $N\le2773$ it follows from the sign of $S_N-1$, computed in exact integers,
at rationals that separate $u_N$ from $\zeta_N/(N-\zeta_N)$ (the odds at which
$z(u)=\zeta_N$). The scripts list these rationals and replay the comparisons with
outward intervals (Section~\ref{sec:verify}). The closest case is $N=218$,
where $\zeta_N$ is more than $1.65\times10^{-5}$ from the certified bracket
of $z_N$. For $N\ge2774$ the sign is Corollary~\ref{cor:every-n}(d).
\end{remark}

The ratios $u_B/u_N$ of the Bessel root of Theorem~\ref{thm:bessel} to the
true odds are $0.9761$, $0.9864$,
$0.9924$ at $N=100,200,400$ and $0.99579$, $0.99769$, $0.99874$ at
$N=800,1600,3200$. So $1-u_B/u_N$ is close to $L/(2N)$, which is $0.0230$ at
$N=100$ and $0.00126$ at $N=3200$. This is in line with
Theorem~\ref{thm:bessel}(iii), which gives a relative error $O(L/N)$ since
$u_N\sim L/N$. Figure~\ref{fig:crossing} compares the exact crossing
with the expansions of Corollary~\ref{cor:every-n} up to $N=10^4$.

\begin{figure}[t]
\centering
\includegraphics{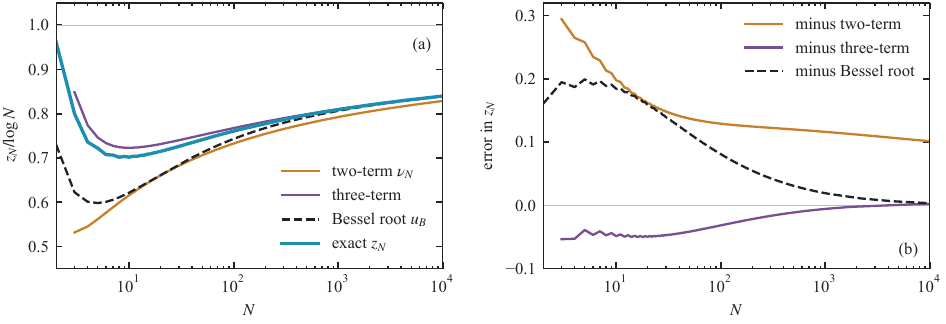}
\caption{(a) The scaled crossing $z_N/\log N$ from the exact run polynomial
(thick line), and the same quantity from the Bessel root $u_B$ of
Theorem~\ref{thm:bessel} (dashed), from the two-term expansion $\nu_N$ of
\eqref{eq:secondorder} and from the three-term expansion
$\nu_N+(\frac14\log L-\frac12c_0+\frac18)/L$ of \eqref{eq:thirdorder}. The
ratio tends to $1$ very slowly. (b) The exact $z_N$ minus each
approximation. The small oscillations for $N$ below about $30$ come from the
parity of $N$.}
\label{fig:crossing}
\end{figure}

\section{The crossing of every bin}\label{sec:allbin}

For each bin $a$ of the free chain we find the value $p_{a,N-a}$ of $p$ at
which bin $a$ is exactly as likely as an endpoint (Section~\ref{sec:periodic}
treats periodic ends). The switching probability $1-p_{a,N-a}$ at the
crossing, shown in Figure~\ref{fig:allbin}(a) for three values of $N$,
decreases as the bin moves in from
the edge, from order $N^{-1/2}$ next to the edge to order $(\log N)/N$ in the
middle.

\subsection{Ordering the crossings}\label{sec:allbin-order}

Let $a,b\ge1$ be integers, $N=a+b$ and $p>0$, and recall that
$S_{a,b}(u)=P_N(a)/P_N(0)$ with $u=(1-p)/p$. Theorem~\ref{thm:pmf} and
$\binom{m-1}{r+1}=\frac{m-1-r}{r+1}\binom{m-1}{r}$ for $m=a,b$ give
\begin{equation}\label{eq:allbin-polynomial}
 S_{a,b}(u)=\sum_{r\ge0}\binom{a-1}{r}\binom{b-1}{r}
 \Bigl(2u^{2r+1}+\frac{N-2-2r}{r+1}u^{2r+2}\Bigr).
\end{equation}
Its coefficients are nonnegative, as $N-2-2r\ge|b-a|$ for $r<\min(a,b)$, and
the linear one is $2$. So there is a unique $u_{a,b}>0$ with
$S_{a,b}(u_{a,b})=1$, and $p_{a,b}=1/(1+u_{a,b})$. As $S_{a,b}(u)\ge2u$, we
have $u_{a,b}\le\frac12$ and $p_{a,b}\in[\frac23,1)$, with $p_{a,b}=\frac23$
only for $a=b=1$.

\begin{proposition}[ordering the crossings]\label{prop:allbin-order}
If $a<b-1$, then
\[
 u_{a+1,b-1}<u_{a,b},\qquad p_{a+1,b-1}>p_{a,b}.
\]
Also $p_{a,b+1}>p_{a,b}$ for all positive $a,b$. In addition,
$S_{a,b+1}\ge S_{a,b}$ coefficientwise, and if $a<b-1$ then
$S_{a+1,b-1}\ge S_{a,b}$ coefficientwise, in each case with at least one
strict coefficient.
\end{proposition}

So a bin closer to the middle, or in a longer row, needs a larger $p$ to fall
to the endpoint height.

\begin{proof}
For fixed $N$ and $r\le a-1$,
$\binom{a-1}{r}\binom{b-1}{r}=(r!)^{-2}\prod_{j=0}^{r-1}(a-1-j)(b-1-j)$, a
product of positive factors. Each factor grows when $(a,b)$ becomes
$(a+1,b-1)$, by $(a-j)(b-2-j)-(a-1-j)(b-1-j)=b-a-1>0$. The multiplier
$(N-2-2r)/(r+1)$ of $u^{2r+2}$ is unchanged, and it is nonnegative where either product is
nonzero, since then $r\le a$ and $N-2-2r\ge b-a-2\ge0$. New summands are
nonnegative, and the cubic coefficient grows strictly, from $2(a-1)(b-1)$ to
$2a(b-2)$. Increasing $b$ with $a$ fixed decreases no coefficient of
\eqref{eq:allbin-polynomial} and increases the quadratic coefficient $N-2$
strictly. A
polynomial with nonnegative coefficients that dominates another
coefficientwise, strictly somewhere, is strictly larger at every $u>0$, so its
solution of $S=1$ is smaller.
\end{proof}

The ordering of the bin heights in Proposition~\ref{prop:adjacent} also
follows from the interior log-concavity of Dekking and
Kong~\cite[Proposition~4.3]{dekkingkong2011}, with the symmetry of the row,
and they remark after that proposition that a sharper proof gives strict
log-concavity. We use the coefficients because they give the strict
inequalities directly and also compare different rows. The formula for
$p_N^{\mathrm{adj}}$ in Proposition~\ref{prop:adjacent} is already implicit in
their work. For $N\ge4$ and $0<p<1$ we have $P_N(1)<P_N(2)$, so
by~\cite[Theorem~4.1]{dekkingkong2011} the row is trimodal exactly when
$P_N(0)>P_N(1)$. Their Example~4.2, with both switching probabilities equal
to $q$ and the stationary start, gives $P_N(1)/P_N(0)=2u+(N-2)u^2$, which is
$S_{1,N-1}(u)$. The same ratio follows from the exact values of $P_N(0)$ and
$P_N(1)$ given by Jensen~\cite[Eq.~(A.1)]{jensen2013}. Zheng, Lin and Zhang~\cite[Eqs.~(7) to (9) of the arXiv version]{zheng2026}
derive the counts of Theorem~\ref{thm:pmf} by Ising's method and the
coefficientwise comparison of Proposition~\ref{prop:allbin-order}, and
conclude that the Landau free energy per spin is smallest at zero
magnetization in the thermodynamic limit. For $p>p_N^{\mathrm{adj}}$ the row is not monotone near its edges, since the
endpoints are more likely than the adjacent bins. So at $p_N$ the bin
probabilities have their
maxima at the edges and in the middle. The next proposition proves this, with
the comparison $p_N>p_N^{\mathrm{adj}}$ for $N\ge4$.

\begin{proposition}[adjacent bins and global extrema]\label{prop:adjacent}
For $N\ge2$ and $0<p<1$, an adjacent bin has the same probability as an
endpoint precisely at
\[
p_N^{\mathrm{adj}}=\frac{1+\sqrt{N-1}}{2+\sqrt{N-1}}.
\]
For $0<p<1$, the endpoints are the least likely bins if
$p<p_N^{\mathrm{adj}}$, bins $1$ and $N-1$ are if $p>p_N^{\mathrm{adj}}$, and
they tie at equality. For $N\ge4$, the central crossing satisfies
$p_N>p_N^{\mathrm{adj}}$. At $p=p_N$, the middle bin and the endpoints attain the
global maximum and bins $1$ and $N-1$ attain the global minimum. For $N=2$
and $N=3$ we have $p_N=p_N^{\mathrm{adj}}$, and every bin ties at the crossing.
\end{proposition}

So for $N\ge4$ the row at the central crossing has the shape of a W. It is high at the two
edges and in the middle, and lowest next to the edges.

\begin{proof}
By \eqref{eq:allbin-polynomial}, $S_{1,N-1}(u)=2u+(N-2)u^2$, with positive
root $u^{\mathrm{adj}}=1/(1+\sqrt{N-1})$. This gives $p_N^{\mathrm{adj}}$ and
the sign change. For $0<p<1$ we have $P_N(k)=P_N(0)S_{k,N-k}(u)$ with
$P_N(0)>0$, so if $1\le k<N-k-1$, then Proposition~\ref{prop:allbin-order} gives
$P_N(k)<P_N(k+1)$. By symmetry the middle is the largest interior bin and bins
$1,N-1$ are the smallest, which gives the claims about the least likely bins.
For $N\ge4$ the cubic coefficient $2(\lfloor N/2\rfloor-1)(\lceil N/2\rceil-1)$
of $S_N$ is positive, so $2u_N+(N-2)u_N^2<S_N(u_N)=1$, hence
$u_N<u^{\mathrm{adj}}$ and $p_N>p_N^{\mathrm{adj}}$. At $p_N$ the middle ties
with the endpoints, which gives the global maximum, and bins $1,N-1$ are below
the endpoints, which gives the global minimum. For $N=2,3$ we have
$S_N=S_{1,N-1}$, and symmetry gives the ties.
\end{proof}

So $1-p_N^{\mathrm{adj}}\sim N^{-1/2}$, much larger than
$1-p_N\sim(\log N)/N$. Away from the ends, the single letter $\sR$ of bin $1$
needs two switches (one into it and one out of it), so $S_{1,N-1}\approx Nu^2$ reaches $1$ only
at $u\approx N^{-1/2}$. (At $p=0$ and $p=1$ there are also ties between bins of
probability $0$.)

\subsection{A Bessel estimate uniform in the bin}\label{sec:allbin-bessel}

Words in bin $a$ with $r+1$ runs of each letter have $2r+1$ switches.
Their contribution to $S_{a,b}$ is about $2u(abu^2)^r/(r!)^2$.
So the natural variable is $abu^2$, and a bin near the
edge acts like the middle of a shorter row with $\sqrt{ab}$ in place of $N/2$.
Put
\[
 s=\sqrt{ab},\quad A=\frac{ab}{a+b},\quad
 c=\frac{a+b}{2\sqrt{ab}},\quad
 \eta=\frac1c,\quad x=su,\quad \mathcal J_c(x)=I_0(2x)+cI_1(2x),
\]
and $\widehat S_{a,b}(u)=2u\mathcal J_c(su)$. Here $c\ge1$ and $0<\eta\le1$, with
$\eta=1$ exactly when $a=b$. This Bessel form is classical. The fixed-scale
law of Herrmann and Vallois~\cite[Remark~3.4]{herrmann2010} with a fair start
has atoms $e^{-\tau}/2$ at $\pm1$ and density
$\frac12\tau e^{-\tau}\{I_0(\tau\rho)+I_1(\tau\rho)/\rho\}$ at magnetization
per spin $m$, where $\rho=\sqrt{1-m^2}$. At $m=(a-b)/N$ we have
$\rho=2s/N$, and with $\tau$ replaced by $Nu$, the density times the bin
width $2/N$, divided by an atom, is $\widehat S_{a,b}(u)$. The next theorem
bounds how far the exact ratio is from it, at every $N$.

\begin{theorem}[uniform estimate for every bin]\label{thm:allbin-bessel}
For all positive integers $a,b$ and all $u>0$,
\begin{equation}\label{eq:allbin-bessel}
 0\le 1-\frac{S_{a,b}(u)}{\widehat S_{a,b}(u)}
 \le \frac{(a+b)u^2}{2}+u
 =\frac{x^2}{2A}+\frac xs.
\end{equation}
In particular, for fixed $N$ and $u$ the bound does not depend on the bin.
The inequality still holds when the right side exceeds $1$, but then it gives no
lower bound on $S_{a,b}$.
\end{theorem}

When $u$ is about $L/N$, the bound is $O(L^2/N)$ for every bin at once. The
proof follows that of Theorem~\ref{thm:bessel}(i), with the two letters
weighted by $a$ and $b$.

\begin{proof}
Let $U_r=U_r(a)$ and $V_r=U_r(b)$ as in Lemma~\ref{lem:clipped}(b), and note
that $1/a+1/b=1/A$. Lemma~\ref{lem:clipped}(a), applied to the $2r$ clipped
factors of $U_rV_r$, gives
\begin{equation}\label{eq:allbin-odd-deficit}
 0\le1-U_rV_r\le\frac{r(r+1)}{2A}.
\end{equation}
Lemma~\ref{lem:clipped}(a) gives $1-U_{r+1}V_r\le\frac{(r+1)(r+2)}{2a}+\frac{r(r+1)}{2b}$,
and the same with $a$ and $b$ exchanged for $U_rV_{r+1}$. So the weighted
average $E_r=(aU_{r+1}V_r+bU_rV_{r+1})/(a+b)$ satisfies
\begin{align}
 0\le1-E_r
 &\le\frac{(r+1)(r+2)}{a+b}
 +\frac{r(r+1)}{2(a+b)}\Bigl(\frac ab+\frac ba\Bigr)\notag\\
 &=\frac{r(r+1)}{2A}+\frac{2(r+1)}{a+b}.\label{eq:allbin-even-deficit}
\end{align}
Since $\binom{a-1}r=a^rU_r/r!$, equation~\eqref{eq:allbin-polynomial} becomes
\begin{equation}\label{eq:allbin-normalized-series}
 S_{a,b}(u)=2u\Bigl[
 \sum_{r\ge0}\frac{x^{2r}}{(r!)^2}U_rV_r
 +c\sum_{r\ge0}\frac{x^{2r+1}}{r!(r+1)!}E_r\Bigr],
\end{equation}
and dropping $U_rV_r$ and $E_r$ gives $\widehat S_{a,b}(u)$. So by
Lemma~\ref{lem:clipped}(c), with $I_j=I_j(2x)$,
$(\widehat S_{a,b}-S_{a,b})/(2u)$ lies between $0$ and
\[
 \frac1{2A}(x^2I_0+xI_1+cx^2I_1)+\frac{2c\,xI_0}{a+b}
 =\Bigl(\frac{x^2}{2A}+\frac xs\Bigr)\mathcal J_c(x),
\]
since $c/s=1/(2A)$ and $2c/(a+b)=1/s$. Divide by $\widehat S_{a,b}/(2u)=\mathcal J_c(x)$.
\end{proof}

In Ising terms the theorem is a bound on the partition function of the
fixed-magnetization ensemble, at every $N$, every coupling and every
magnetization $|\mu|<N$.

\begin{corollary}[fixed-magnetization weight]\label{cor:fixed-M}
Let $|\mu|<N$ with $\mu\equiv N\pmod2$, and let $Z_N(\mu)$ be the total
weight of the configurations of the free chain with magnetization $\mu$, as
in Section~\ref{sec:setting}. Put $\rho=\sqrt{1-\mu^2/N^2}$ and
$v=Ne^{-2J}\rho$. Then for every real $J$,
\[
Z_N(\mu)=2e^{(N-3)J}\Bigl(I_0(v)+\frac{I_1(v)}\rho\Bigr)(1-\theta),\qquad
0\le\theta\le\frac N2e^{-4J}+e^{-2J}.
\]
\end{corollary}

\begin{proof}
Put $a=(N+\mu)/2$ and $b=(N-\mu)/2$. Then $Z_N(\mu)=e^{(N-1)J}S_{a,b}(u)$ with
$u=e^{-2J}$, and $2s=2\sqrt{ab}=N\rho$, so $2su=v$ and $c=1/\rho$. Now
$\widehat S_{a,b}(u)=2u(I_0(v)+I_1(v)/\rho)$ and $2ue^{(N-1)J}=2e^{(N-3)J}$, and
Theorem~\ref{thm:allbin-bessel} bounds $\theta$.
\end{proof}

\begin{remark}[the sign of the first correction]\label{rem:sign-correction}
Theorem~\ref{thm:allbin-bessel} shows that the exact weight lies below its
Bessel form, so the first correction to the Bessel form is negative. Its size
is of the order of the bound, since the clipped products $U_rV_r$ fall short
of $1$ by about $r(r+1)/(2A)$. In the published Appendix~D of Dantchev,
Tonchev and Rudnick~\cite[p.~19, Eqs.~(D.4) and (D.6)]{dantchev2023}, the
correction to the Bessel form of ${}_2F_1(-a,-b;1;z)$ has the opposite sign
(their $a$ and $b$ are our $a-1$ and $b-1$). Their
Eq.~(D.4) writes $\binom an n!$ as $(a+n-1)!/(a-1)!$, while it is the falling
factorial $a!/(a-n)!$, and with the falling factorial the sign of the
correction is reversed. For example, with their parameters $a=b=200$ and $abz=25$, the
exact ${}_2F_1(-a,-b;1;z)$ is $9.5$ percent below $I_0(2\sqrt{abz})$, while
their Eq.~(D.6) is $10.1$ percent above it and the same formula with the
opposite sign is within $0.8$ percent of the exact value. This concerns
only the sign of that correction.
\end{remark}

To turn Theorem~\ref{thm:allbin-bessel} into a bound on the crossing, let $x_{a,b}=su_{a,b}$ and let
$\hat x_{a,b}$ be the positive root of
\begin{equation}\label{eq:allbin-bessel-root}
 G_\eta(\hat x_{a,b})=A,\qquad
 G_\eta(x)=x\{I_1(2x)+\eta I_0(2x)\}.
\end{equation}
Since $c\eta=1$ and $2c/s=1/A$, we have $\widehat S_{a,b}(u)=G_\eta(su)/A$.
Differentiating term by term gives
\begin{equation}\label{eq:allbin-logderivative}
 G_\eta'(x)-2G_\eta(x)
 =2x(1-\eta)\{I_0(2x)-I_1(2x)\}+\eta I_0(2x)\ge0,
\end{equation}
since $I_0(2x)-I_1(2x)=\frac1\pi\int_0^\pi e^{2x\cos t}(1-\cos t)\,dt>0$
by \cite[Eq.~10.32.3]{dlmf}. So $(\log G_\eta)'\ge2$ for $x>0$ and
$0\le\eta\le1$, and $\hat x_{a,b}$ is unique.

\begin{corollary}[quantitative crossing bound]\label{cor:allbin-root-bound}
We have $x_{a,b}\ge \hat x_{a,b}$. If
\[
 \delta_{a,b}=\frac{x_{a,b}^2}{2A}+\frac{x_{a,b}}s<1,
\]
then
\begin{equation}\label{eq:allbin-root-bound}
 0\le x_{a,b}-\hat x_{a,b}\le-\tfrac12\log(1-\delta_{a,b}).
\end{equation}
\end{corollary}

\begin{proof}
At
the crossing $S_{a,b}=1$, so Theorem~\ref{thm:allbin-bessel} gives
$A\le G_\eta(x_{a,b})\le A/(1-\delta_{a,b})$. As $G_\eta(\hat x_{a,b})=A$, this gives
$x_{a,b}\ge \hat x_{a,b}$, and integrating $(\log G_\eta)'\ge2$ from $\hat x_{a,b}$ to
$x_{a,b}$ gives \eqref{eq:allbin-root-bound}. This is the argument of
Theorem~\ref{thm:bessel}(ii). For even $N$ and $a=b$ the two statements
coincide, since then $\eta=1$, $s=h$ and $2G_1=F$, so $\hat x_{a,b}=x_B$.
\end{proof}

\begin{remark}[Jacobi polynomials]\label{rem:jacobi}
The run sums of this paper can be written in terms of Jacobi polynomials, and
their Bessel and Laguerre forms are limits of Jacobi polynomials near the
point $1$.
For $a\le b$ and $0\le t<1$,
\[
 \sum_{r\ge0}\binom{a-1}{r}\binom{b-1}{r}t^r
 ={}_2F_1(1-a,1-b;1;t)
 =(1-t)^{a-1}P_{a-1}^{(0,\,b-a)}\Bigl(\frac{1+t}{1-t}\Bigr),
\]
and for $b>a$ and $0\le u<1$,
\[
 \sum_{r\ge0}\binom{a-1}{r}\binom{b-1}{r+1}u^{2r+2}
 =\frac{(b-1)u^2}{a}(1-u^2)^{a-1}P_{a-1}^{(1,\,b-a-1)}\Bigl(\frac{1+u^2}{1-u^2}\Bigr).
\]
For fixed $\alpha,\beta>-1$, the Mehler--Heine limit
\cite[Eq.~18.11.5]{dlmf} holds uniformly on bounded sets of complex
arguments~\cite[Theorem~8.1.1]{szego1975}. Taken at the argument $iw$, it gives
$n^{-\alpha}P_n^{(\alpha,\beta)}(1+w^2/(2n^2))
\to(w/2)^{-\alpha}I_\alpha(w)$ as $n\to\infty$, uniformly for bounded $w\ge0$,
with the value at $w=0$ taken by continuity.
Elliott~\cite[Section~3, Eq.~(3.10)]{elliott1971} gives formal modified-Bessel
expansions in the plane cut along $(-\infty,1]$, away from $-1$, for fixed
parameters with $\alpha\ge0$. They apply on the real axis to the right of $1$,
where the arguments above lie, but he defers truncation-error bounds to a
proposed later paper~\cite[Section~1]{elliott1971}.
Dunster~\cite[Section~2, Eqs.~(2.43) to~(2.50)]{dunster1999} gives
approximations by Whittaker functions on unbounded complex domains when the
second Jacobi parameter is large and the degree is at most a fixed multiple
of it. In the ultraspherical case $\alpha=\beta$, which includes the
first sum above when $a=b$, he gives Bessel approximations that hold up
to and beyond the point $1$~\cite[Section~5, Eqs.~(5.11), (5.15)
and~(5.16)]{dunster1999}. He states that the general theories he applies
give explicit error bounds for all of these approximations~\cite[p.~287]{dunster1999}. They use transformed arguments
and prefactors. Our estimates instead compare the run sums with the stated
Bessel forms, with a sign and explicit constants, for every finite chain
and every interior bin.
For fixed $a$ and $y\ge0$, put $u^2=y/b$ and $\beta=b-a-1$. The argument of
the second Jacobi polynomial is $1-2\tilde x/\beta$, where
$\tilde x=-\beta y/(b-y)\to-y$ as $b\to\infty$.
The fixed-degree Jacobi--Laguerre limit~\cite[Eq.~18.7.21]{dlmf} is locally
uniform in its argument, by coefficient convergence of polynomials of fixed
degree. It makes the last sum tend to
$(y/a)L_{a-1}^{(1)}(-y)=H_a(y)$ of \eqref{eq:edge-polynomial}.
The odd-run contribution, including its factor $2u$, and the other even-run
contribution are respectively $O_a(b^{-1/2})$ and $O_a(b^{-1})$, uniformly
for $y$ in compact sets. Temme~\cite[Section~4.2]{temme1990} also compares
Jacobi with Laguerre polynomials, and Gil, Segura and
Temme~\cite[Section~2 and Remark~1]{gil2018} expand
$P_n^{(\alpha,\beta)}(1-2x/\beta)$ in inverse powers of $\beta$ with
Laguerre coefficients, uniformly for bounded $n$, $\alpha$ and $x$ (complex
$x$ included). This covers the fixed degree $a-1$ of
Section~\ref{sec:allbin-edge}, while Theorem~\ref{thm:allbin-expansion} lets
the degree grow, uniformly in $b\ge a$. Also, Theorem~\ref{thm:edge-laguerre}
is about the root of $H_a(y)=1$, a level of the Laguerre polynomial at the
negative argument $-y$, and not about its zeros. Here $b-a$ can be as large as
$N-2$, and Theorem~\ref{thm:allbin-bessel} gives an explicit relative error at
every $N$.
\end{remark}

\subsection{Every growing distance from an edge}\label{sec:allbin-growing}

The Bessel root satisfies an equation of the form
$z+\frac12\log z=\ell+C+B/z+\dots$. The next lemma inverts such equations.

\begin{lemma}[inverting $z+\frac12\log z$]\label{lem:inversion}
Fix $K>0$. There is a constant $C_K$ with the following property. Let $z\ge1$
and $\ell\ge2$, and suppose that $z+\tfrac12\log z=\ell+C+B/z+\omega$ with
$|C|,|B|\le K$ and $|\omega|\le1$. Then
\[
\Bigl|z-\ell+\tfrac12\log\ell-C-\frac{\frac14\log\ell-\frac12C+B}{\ell}\Bigr|
\le C_K\Bigl(\frac{(\log\ell)^2}{\ell^2}+|\omega|\Bigr).
\]
\end{lemma}

\begin{proof}
Put $w=z-\ell$. As $z\ge1$, $\log z\ge0$ and $|B/z|\le K$, so $z\le\ell+2K+1$
and $|w|\le2K+1+\frac12\log(\ell+2K+1)=O_K(\log\ell)$. This proves the claim
for $\ell$ bounded in terms of $K$, since then both sides are $O_K(1)$ and
$(\log\ell)^2/\ell^2$ is bounded below. For larger $\ell$ we have
$|w|\le\ell/2$, so $\log z=\log\ell+w/\ell+O((w/\ell)^2)$ and
$B/z=B/\ell+O_K(|w|/\ell^2)$, and
$w=C-\frac12\log\ell-w/(2\ell)+B/\ell+\omega+O_K((\log\ell)^2/\ell^2)$. This
gives first $w=d+O_K(\log\ell/\ell+|\omega|)$ with $d=C-\frac12\log\ell$, and
then, putting this back into the term $w/(2\ell)$,
$w=d+(B-d/2)/\ell+O_K((\log\ell)^2/\ell^2+|\omega|)$.
\end{proof}

\begin{theorem}[uniform crossing expansion]\label{thm:allbin-expansion}
Let $1\le a\le b$, let $A$ and $\eta$ be as in Section~\ref{sec:allbin-bessel}, and put
\[
 \ell=\log A,\qquad
 C(\eta)=\tfrac12\log(8\pi)-\log(1+\eta),\qquad
 B(\eta)=\frac{3-\eta}{8(1+\eta)}.
\]
As $a\to\infty$, uniformly over $b\ge a$,
\begin{align}
 2\sqrt{ab}\,u_{a,b}
 ={}&\ell-\tfrac12\log\ell+C(\eta)\notag\\
 &+\frac{\tfrac14\log\ell-\tfrac12C(\eta)+B(\eta)}{\ell}
 +O\left(\frac{(\log\ell)^2}{\ell^2}+\frac{(\log a)^2}{a}\right).
 \label{eq:allbin-expansion}
\end{align}
That is, there are $a_0$ and $K'$ such that the error term is at most $K'$
times the bracket for all $a\ge a_0$ and all $b\ge a$. The same holds with
$1-p_{a,b}$ in place of $u_{a,b}$.
\end{theorem}

Here $A=ab/(a+b)$ is half the harmonic mean of $a$ and $b$. There is no
condition on how $a$ compares with $N$, so the theorem covers a bin at
distance $\sqrt N$ from the edge as well as the middle. The proof uses
Corollary~\ref{cor:allbin-root-bound} to pass to the Bessel root $\hat x_{a,b}$. The
large-argument form of $I_0$ and $I_1$ then turns
\eqref{eq:allbin-bessel-root} into an equation that Lemma~\ref{lem:inversion}
solves.

\begin{proof}
Here $a/2\le A\le a$, $s\ge a$ and $0<\eta\le1$.

\emph{Step 1. The crossing is at $x_{a,b}\le\log a$.} At $x=\log a$ the bound
of Theorem~\ref{thm:allbin-bessel} is at most $2(\log a)^2/a$, while
$\widehat S_{a,b}=G_\eta(x)/A\ge xI_1(2x)/a\to\infty$, both uniformly in $b$. So $S_{a,b}>1$ there, and since
$S_{a,b}$ increases, $x_{a,b}\le\log a$ for all $a\ge a_0$ and all $b\ge a$. Then
$\delta_{a,b}\le2(\log a)^2/a$, and Corollary~\ref{cor:allbin-root-bound}
gives
\begin{equation}\label{eq:allbin-root-error}
 x_{a,b}-\hat x_{a,b}=O((\log a)^2/a).
\end{equation}

\emph{Step 2. The Bessel root.} Write $z=2\hat x_{a,b}$, so $z\to\infty$ as
$A\to\infty$, uniformly in $\eta$. By the large-argument form of $I_0$ and
$I_1$ \cite[Eq.~10.40.1]{dlmf},
$I_1(z)+\eta I_0(z)=(1+\eta)e^z(2\pi z)^{-1/2}(1-B(\eta)/z+O(z^{-2}))$
uniformly in $0\le\eta\le1$. Taking logarithms in
\eqref{eq:allbin-bessel-root} gives
$z+\frac12\log z=\ell+C(\eta)+B(\eta)/z+O(z^{-2})$. As $|C(\eta)|\le2$ and
$0<B(\eta)\le\frac38$, Lemma~\ref{lem:inversion} with $K=2$ expands $z$.

\emph{Step 3. Back to the crossing.} By \eqref{eq:allbin-root-error} the
expansion moves to $2x_{a,b}=2\sqrt{ab}\,u_{a,b}$. Finally
$2s(u-u/(1+u))\le2x^2/s=O((\log a)^2/a)$ at $u=u_{a,b}$, which gives the same
expansion for $1-p_{a,b}=u_{a,b}/(1+u_{a,b})$.
\end{proof}

Hence, uniformly in $b$,
\begin{equation}\label{eq:allbin-leading}
 1-p_{a,b}\sim\frac{\log a}{2\sqrt{ab}}
 \qquad(a\to\infty,\ b\ge a).
\end{equation}
If $a\to\infty$ with $a=o(b)$, then $\eta\to0$ and
$\log A=\log a+o(1)$, so
\begin{equation}\label{eq:allbin-growing-edge}
 2\sqrt{ab}(1-p_{a,b})
 =\log a-\tfrac12\log\log a+\tfrac12\log(8\pi)+o(1).
\end{equation}
If $a/N\to\alpha\in(0,\frac12]$, however slowly, then
\[
 2\sqrt{a(N-a)}(1-p_{a,N-a})
 =\log N-\tfrac12\log\log N+\log\{\alpha(1-\alpha)\}+C(2\sqrt{\alpha(1-\alpha)})+o(1).
\]
At the center for even $N$, $a=b=N/2$ gives $A=N/4$ and $\eta=1$, and changing from
$\log A$ to $\log N$ recovers \eqref{eq:thirdorder}.

\begin{figure}[t]
\centering
\includegraphics{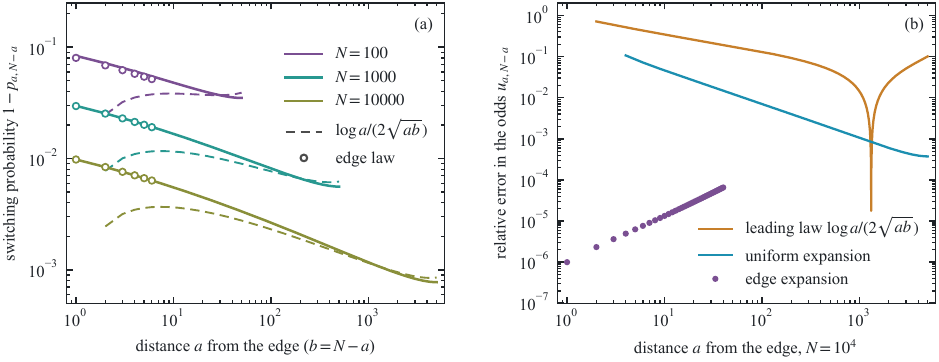}
\caption{(a) The switching probability $1-p_{a,N-a}$ at the crossing of
each bin, from the exact run polynomial, against its distance $a$ from the
nearer edge, for $N=100$, $1000$ and $10^4$. The dashed lines are the leading
law \eqref{eq:allbin-leading}. It is asymptotic in $a$, so it departs from
the exact curves near the edge. The circles are the first two terms of the
edge law
\eqref{eq:edge-q-expansion} for $a\le6$. The right end of each line is the
central crossing $1-p_N$. (b) For $N=10^4$, the absolute relative error in the odds
$u_{a,N-a}$ of the leading law $(\log a)/(2\sqrt{ab})$ with $b=N-a$, of the main terms of
the uniform expansion \eqref{eq:allbin-expansion} for $a\ge4$, and of the
first three terms of the edge expansion \eqref{eq:edge-crossing-expansion}
for $a\le40$. The dip near $a=1300$ is where the error of the leading law changes sign.}
\label{fig:allbin}
\end{figure}

\subsection{Fixed bins and the Laguerre edge law}\label{sec:allbin-edge}

In Ising terms, bin $a$ with $a$ fixed is magnetization $2a-N$, a few up spins
in a sea of down spins, and we ask when it is as likely as full order. Let
$b\to\infty$. A run of letters $\sR$ away from the
ends needs two switches, which gives a factor $u^2$, and it can go in about $b$
places. So the right variable is $y=bu^2$. Define
\begin{equation}\label{eq:edge-polynomial}
 H_a(y)=\sum_{r=0}^{a-1}\binom{a-1}{r}\frac{y^{r+1}}{(r+1)!}
 =\frac ya L_{a-1}^{(1)}(-y),
\end{equation}
with the standard Laguerre polynomial $L_{a-1}^{(1)}$. In a field these
polynomials give the classical P\'olya--Aeppli law
(Section~\ref{sec:field}). Let $y_a$ be the
positive root of $H_a(y_a)=1$ and $c_a=\sqrt{y_a}$. Then $y_1=1$,
$y_2=\sqrt3-1$, and $y_a$ decreases strictly in $a$, by coefficient
comparison.

\begin{theorem}[fixed-edge expansion]\label{thm:edge-laguerre}
For each fixed positive integer $a$, as $b\to\infty$,
\begin{equation}\label{eq:edge-crossing-expansion}
 u_{a,b}=\frac{c_a}{\sqrt b}-\frac1b
 +\frac{\gamma_a}{b^{3/2}}+O_a(b^{-2}),
\end{equation}
where
\[
 \gamma_a=
 \frac{1+y_a+(y_a+y_a^2/2)H_a''(y_a)/H_a'(y_a)}{2c_a}.
\]
Thus the second coefficient in the odds expansion is $-1$ for every fixed
bin. For the switching probability,
\begin{equation}\label{eq:edge-q-expansion}
 1-p_{a,b}=\frac{c_a}{\sqrt b}-\frac{1+c_a^2}{b}+O_a(b^{-3/2}).
\end{equation}
\end{theorem}

For bin $1$, where $y_1=1$ and $H_1''=0$, this gives
$u_{1,b}=b^{-1/2}-b^{-1}+b^{-3/2}+\dots$, the expansion of the exact root
$1/(1+\sqrt b)$ of Proposition~\ref{prop:adjacent}.

\begin{proof}
Put $\theta=b^{-1/2}$, $y=bu^2$ and $V_r=U_r(b)$. For $b\ge a$,
\eqref{eq:allbin-polynomial} reads
\begin{align}
 S_{a,b}(u)={}&2\sqrt{y/b}\sum_r\binom{a-1}{r}\frac{y^r}{r!}V_r\notag\\
 &+\frac yb\sum_r\binom{a-1}{r+1}\frac{y^r}{r!}V_r
 +y\sum_r\binom{a-1}{r}\frac{y^r}{(r+1)!}V_{r+1}.
 \label{eq:edge-exact-series}
\end{align}
These sums have at most $a$ terms, and with $V_r=1$ they are $H_a'(y)$,
$H_a''(y)$ and $H_a(y)$. Expanding $V_r=1-r(r+1)/(2b)+O_a(b^{-2})$ and using
$\sum_r(r+1)(r+2)\binom{a-1}{r}\frac{y^{r+1}}{(r+1)!}=y^2H_a''(y)+2yH_a'(y)$
gives, uniformly near $y=y_a$,
\begin{align}
 S_{a,b}(u)={}&H_a(y)+2\sqrt y\,H_a'(y)\theta\notag\\
 &+\{(y-y^2/2)H_a''(y)-yH_a'(y)\}\theta^2+O_a(\theta^3).
 \label{eq:edge-implicit-expansion}
\end{align}
With $b=\theta^{-2}$ the right side of \eqref{eq:edge-exact-series} is
analytic in $(\theta,y)$ near $(0,y_a)$, and $H_a'(y_a)>0$. So the
implicit-function theorem gives an analytic solution
$y(\theta)=y_a+d_1\theta+d_2\theta^2+O_a(\theta^3)$ of $S_{a,b}=1$,
and by the uniqueness of the crossing it is the value of $y$ at $u=u_{a,b}$.
Matching coefficients in \eqref{eq:edge-implicit-expansion} gives $d_1=-2c_a$
and $d_2=2+y_a+(y_a+y_a^2/2)H_a''(y_a)/H_a'(y_a)$. Expanding $u=\theta\sqrt y$
and then $u/(1+u)$ proves both expansions.
\end{proof}

The next theorem shows that for large $a$ the edge constant $2\sqrt{ay_a}$,
the limit of $2\sqrt{ab}\,u_{a,b}$ as $b\to\infty$, has the expansion of the
growing-distance law \eqref{eq:allbin-growing-edge}.

\begin{theorem}[matching the edge and Bessel regimes]\label{thm:edge-matching}
For every positive integer $a$ and every $x>0$,
\begin{equation}\label{eq:laguerre-bessel-bound}
 0\le1-\frac{aH_a(x^2/a)}{xI_1(2x)}\le\frac{x^2}{2a}.
\end{equation}
Consequently, with $\ell=\log a$ and with $C(0)=\tfrac12\log(8\pi)$ from
Theorem~\ref{thm:allbin-expansion}, as $a\to\infty$,
\begin{align}
 2\sqrt{a y_a}
 ={}&\ell-\tfrac12\log\ell+C(0)\notag\\
 &+\frac{\tfrac14\log\ell-\tfrac12C(0)+3/8}{\ell}
 +O\left(\frac{(\log\ell)^2}{\ell^2}+\frac{(\log a)^2}{a}\right).
 \label{eq:edge-matching-expansion}
\end{align}
\end{theorem}

\begin{proof}
With $U_r=U_r(a)$, $aH_a(x^2/a)=x\sum_{r\ge0}U_rx^{2r+1}/(r!(r+1)!)$, so
Lemma~\ref{lem:clipped}(b) and the second identity of
Lemma~\ref{lem:clipped}(c) give \eqref{eq:laguerre-bessel-bound}. Put
$\omega_a=\sqrt{ay_a}$, so $aH_a(\omega_a^2/a)=a$, and $\omega_a<\log a$ for large $a$ by
\eqref{eq:laguerre-bessel-bound}. So $a\le G_0(\omega_a)\le a/(1-\omega_a^2/(2a))$, and
as in Corollary~\ref{cor:allbin-root-bound}, $\omega_a-\hat\omega_a=O((\log a)^2/a)$, where
$G_0(\hat\omega_a)=a$. So $z=2\hat\omega_a$ satisfies
$z+\frac12\log z=\ell+C(0)+3/(8z)+O(z^{-2})$ by \cite[Eq.~10.40.1]{dlmf}, as in
the proof of Theorem~\ref{thm:allbin-expansion} with $\eta=0$ and $A=a$, and
Lemma~\ref{lem:inversion} gives the expansion.
\end{proof}

Figure~\ref{fig:allbin} compares these laws with the exact crossings.

\subsection{The moving boundary layer}\label{sec:allbin-layer}

The crossings also give the shape of a whole row at odds $u=v_N>0$. Between
the central crossing ($v_N$ about $(\log N)/N$) and the adjacent one ($v_N$
about $N^{-1/2}$), the middle is above the endpoints and a layer of bins next
to each edge is below them. (This is a layer of bins near each end of the
row, that is, in the magnetization, not a region of the chain.) The width of
this layer is
\[
 \mathcal L_N=\#\{1\le k\le\lfloor N/2\rfloor:P_N(k)<P_N(0)\}.
\]

\begin{corollary}[boundary-layer width]\label{cor:boundary-layer}
Suppose that, as $N\to\infty$,
\[
 \Omega_N=\frac1{\sqrt N\,v_N}\longrightarrow\infty,
 \qquad \frac{\log\Omega_N}{Nv_N}\longrightarrow0.
\]
Then
\begin{equation}\label{eq:boundary-layer-width}
 \mathcal L_N\sim\Omega_N^2\log^2\Omega_N
 =\frac{\log^2(1/(\sqrt N\,v_N))}{Nv_N^2}.
\end{equation}
In particular, if $v_N=N^{-\kappa}$ with $1/2<\kappa<1$, then
\[
 \mathcal L_N\sim(\kappa-\tfrac12)^2N^{2\kappa-1}\log^2N.
\]
By symmetry, the total number of interior bins below the endpoint height
is $2\mathcal L_N$ for all sufficiently large $N$.
\end{corollary}

\begin{proof}
Set $d_N=\Omega_N^2\log^2\Omega_N$. Then $d_N\to\infty$ and
$d_N/N=(\log\Omega_N/(Nv_N))^2\to0$. Fix $\epsilon\in(0,1)$ and integers
$a_\pm=(1\pm\epsilon)d_N+O(1)$. Then $\log a_\pm\sim2\log\Omega_N$ and
$\sqrt{a_\pm(N-a_\pm)}\,v_N\sim\sqrt{1\pm\epsilon}\,\log\Omega_N$, so
\eqref{eq:allbin-leading} gives
$u_{a_\pm,N-a_\pm}/v_N\to1/\sqrt{1\pm\epsilon}$. Bin $k$ is below the
endpoint exactly when $v_N<u_{k,N-k}$, and $u_{k,N-k}$ decreases strictly for
$k\le N/2$ by Proposition~\ref{prop:allbin-order}. So
$a_-\le\mathcal L_N<a_+$ for large $N$, and we let $\epsilon\downarrow0$.
Finally $\mathcal L_N=o(N)$, so the reflected intervals are disjoint.
\end{proof}

Figure~\ref{fig:layer} shows one row inside this regime and the width
$\mathcal L_N$ for three exponents $\kappa$. If $\mathcal L_N$ is instead divided by the
distance $a$ at which the main terms of \eqref{eq:allbin-expansion}, divided
by $2\sqrt{a(N-a)}$, equal $v_N$, the ratio is within $0.013$ of $1$ for every $N$ in panel (b).

\begin{figure}[t]
\centering
\includegraphics{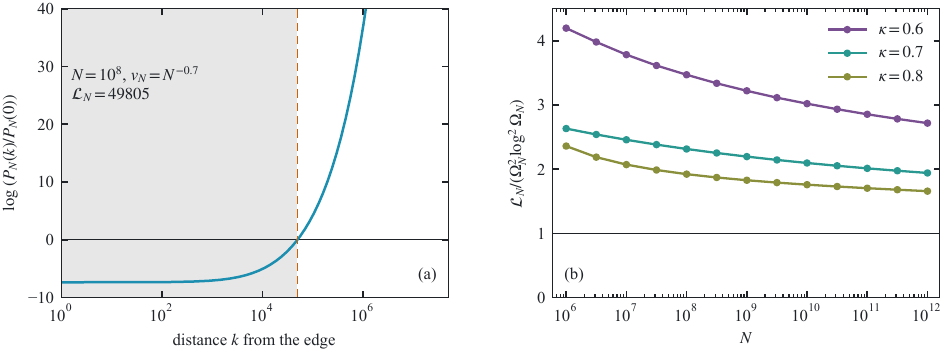}
\caption{(a) The logarithm $\log(P_N(k)/P_N(0))$ of the ratio of each bin to an endpoint,
from the exact run polynomial, against the distance $k$ from the nearer edge,
at $N=10^8$ and odds $v_N=N^{-0.7}$. The shaded bins are below the endpoint
height, and there are $\mathcal L_N=49805$ of them in the left half. (b) The
width $\mathcal L_N$ divided by its leading term $\Omega_N^2\log^2\Omega_N$
of \eqref{eq:boundary-layer-width}, for $v_N=N^{-\kappa}$ ($\kappa=0.6$,
$0.7$ and $0.8$) and $10^6\le N\le10^{12}$. The ratios decrease toward $1$ only slowly.}
\label{fig:layer}
\end{figure}

\section{Periodic boundaries and the crossing window}\label{sec:periodicwindow}

\subsection{Periodic boundaries}\label{sec:periodic}

We now close the sequence of steps into a loop, so that the last step also
interacts with the first. With this closing bond the spins form the periodic
Ising chain, with weight proportional to
$\exp(J\sum_{i=1}^N\sigma_i\sigma_{i+1})$, $\sigma_{N+1}=\sigma_1$ and again
$u=e^{-2J}$. Garc\'{\i}a-Pelayo~\cite[Eq.~(15)]{garciapelayo2009} gives its exact
magnetization law in hypergeometric form. Antal, Droz and R\'acz~\cite{antal2004} give its magnetization law
at a fixed switching scale $\tau$, and Dantchev and Rudnick~\cite{dantchev2022} give exact
coefficients. We find how much the closing bond moves the central crossing,
and then describe the row near the crossing for both boundary conditions.

Write $S^{\mathrm F}_{N,k}=S_{k,N-k}$ and $S^{\mathrm P}_{N,k}$ for the
bin-to-endpoint ratios for free and periodic ends. Put $a=k$, $b=N-k$
and $s=\sqrt{ab}$. Counting cyclic words by their runs gives
\begin{equation}\label{eq:periodicratio}
S^{\mathrm P}_{N,k}(u)=\sum_{r=1}^{\min(a,b)}\frac Nr
 \binom{a-1}{r-1}\binom{b-1}{r-1}u^{2r}.
\end{equation}
To see this, note that a cyclic word with $a$ letters $\sR$ and $b$ letters
$\sL$ switches an even number $2r$ of times and has $r$ runs of each letter.
Choose compositions of $a$ and of $b$ into $r$ positive runs and a position for
the start of a distinguished run of letters $\sR$, and divide by the $r$ choices of
that run. This is the count of~\cite[Appendix~B, Eqs.~(B2') to~(B4)]{garciapelayo2009},
\cite[Eq.~(30)]{antal2004} and~\cite[Eq.~(2.9)]{dantchev2022}. The normalization of the periodic
measure cancels in the ratio. In
particular $S^{\mathrm P}_{N,1}=Nu^2$, which gives the adjacent crossing
$J=(\log N)/4$ of Stepanyan et~al.\@~\cite{stepanyan2024}. Every interior
ratio increases strictly from $0$ to $\infty$ in $u>0$, and the ratios
increase strictly toward the middle, apart from the symmetric middle tie, by
the coefficient comparison in the proof of Proposition~\ref{prop:allbin-order},
applied to \eqref{eq:periodicratio}.

\begin{theorem}[periodic relative error]\label{thm:periodic-bessel}
For $N\ge2$, $1\le k\le N-1$, and $u>0$, define
\[
\widehat S^{\mathrm P}_{N,k}(u)=\frac{Nu}{\sqrt{k(N-k)}}
 I_1\!\left(2u\sqrt{k(N-k)}\right).
\]
Then
\begin{equation}\label{eq:periodicbound}
0\le1-\frac{S^{\mathrm P}_{N,k}(u)}{\widehat S^{\mathrm P}_{N,k}(u)}
\le\frac{Nu^2}{2}.
\end{equation}
\end{theorem}

So for periodic ends the ratio of every bin to an endpoint is close to
a single Bessel function $I_1$, with the same kind of error as in Theorem~\ref{thm:allbin-bessel}.

\begin{proof}
Put $x=su$, $A=ab/N$, $U_r=U_r(a)$ and $V_r=U_r(b)$. Shifting the index in
\eqref{eq:periodicratio} gives
$S^{\mathrm P}_{N,k}(u)=(Nu/s)\sum_{r\ge0}U_rV_rx^{2r+1}/(r!(r+1)!)$, so the
ratio to $\widehat S^{\mathrm P}_{N,k}$ is an average of the $U_rV_r$ with
weights $x^{2r+1}/(r!(r+1)!)$, whose sum is $I_1(2x)$. By
\eqref{eq:allbin-odd-deficit}, $0\le1-U_rV_r\le r(r+1)/(2A)$, and the second
identity of Lemma~\ref{lem:clipped}(c) bounds the deficit by
$x^2/(2A)=Nu^2/2$.
\end{proof}

A cyclic word switches an
even number of times, so only $I_1$ appears in the periodic approximation.
The free words with an odd number of switches start and end on
different letters, and they give the extra $I_0$ term of the free one.
As $I_0\approx I_1$ for large arguments, the free ratio is about twice the
periodic one near the crossing. Both grow like $e^{Nu}$, so the periodic
crossing needs $Nu$ about $\log2$ larger. Antal, Droz and
R\'acz~\cite[Eqs.~(37) and~(47)]{antal2004} work at a fixed value of $\tau=Nq$, their scaling variable
(Section~\ref{sec:setting}). There the endpoint atoms are
$1/(2\cosh \tau)$ (periodic) and $e^{-\tau}/2$ (free), which again differ by a
factor of about $2$ for large $\tau$, while their densities at zero magnetization
agree to leading order as $\tau\to\infty$.

\begin{theorem}[periodic crossing and boundary shift]\label{thm:periodic-crossing}
Let $u_N^{\mathrm P}$ be the positive root of $S^{\mathrm P}_{N,m_N}(u)=1$,
let $p_N^{\mathrm P}=1/(1+u_N^{\mathrm P})$, and let $u_N^{\mathrm F}=u_N$ be
the free central root. Put $c_{\mathrm P}=\tfrac12\log(\pi/2)$. As
$N\to\infty$,
\begin{align}
Nu_N^{\mathrm P}
 &=L-\tfrac12\log L+c_{\mathrm P}
 +\frac{\tfrac14\log L-\tfrac12c_{\mathrm P}+\tfrac38}{L}
 +O\!\left(\frac{(\log L)^2}{L^2}\right),\label{eq:periodicasym}\\
N(u_N^{\mathrm P}-u_N^{\mathrm F})
 &=\log2+\frac{\tfrac14-\tfrac12\log2}{L}
 +O\!\left(\frac{(\log L)^2}{L^2}\right).\label{eq:boundaryshift}
\end{align}
Replacing $Nu_N^{\mathrm P}$ by $N(1-p_N^{\mathrm P})$ preserves
\eqref{eq:periodicasym}. At the periodic crossing the coupling
$J_N^{\mathrm P}=-\frac12\log u_N^{\mathrm P}$ satisfies
\begin{equation}\label{eq:isingtemperature}
J_N^{\mathrm P}
=\tfrac12(\log N-\log\log N)
+\frac{\tfrac14\log\log N-\tfrac12c_{\mathrm P}}{\log N}
+O\!\left(\frac{(\log\log N)^2}{(\log N)^2}\right).
\end{equation}
\end{theorem}

Since $c_{\mathrm P}=c_0+\log2$, the periodic expansion is the free one
\eqref{eq:thirdorder} with $c_0$ replaced by $c_0+\log2$, and with $\frac38$ in
place of $\frac18$ in the $1/L$ term.

\begin{proof}
\emph{Step 1. The first order.} At the middle $2s=N+O(1/N)$. At
$u=(1\pm\theta)L/N$ with fixed $\theta\in(0,1)$ the bound of
Theorem~\ref{thm:periodic-bessel} tends to $0$, and the large-argument form of
$I_1$ makes $\widehat S^{\mathrm P}$ tend to $\infty$ with the plus sign and
to $0$ with the minus sign. Since $S^{\mathrm P}_{N,m_N}$ increases in $u$, we
get $Nu_N^{\mathrm P}\sim L$.

\emph{Step 2. The crossing equation.} Write $z=Nu_N^{\mathrm P}$. At the root
$S^{\mathrm P}=1$, so $\log\widehat S^{\mathrm P}=O(L^2/N)$ by
\eqref{eq:periodicbound}. Replacing the argument $2su$ by $z$, and the factor
$Nu/s$ by $2z/N$, changes $\log\widehat S^{\mathrm P}$ by $O(L/N^2)$.
Here $(\log I_1)'=I_0/I_1-1/z=O(1)$ for large positive arguments
by~\cite[Eqs.~10.29.2, 10.40.1]{dlmf}. Also
$I_1(z)=e^z(2\pi z)^{-1/2}(1-3/(8z)+O(z^{-2}))$~\cite[Eq.~10.40.1]{dlmf}. So
\[
z+\tfrac12\log z=L+c_{\mathrm P}+\frac3{8z}+O(z^{-2}+L^2/N),
\]
and Lemma~\ref{lem:inversion} with $\ell=L$, $C=c_{\mathrm P}$ and
$B=\frac38$ gives \eqref{eq:periodicasym}. As $N(u-u/(1+u))=O(L^2/N)$ for
$Nu=O(L)$, it also holds for $N(1-p_N^{\mathrm P})$.

\emph{Step 3. The shift and the coupling.} By \eqref{eq:zN},
$Nu_N=z_N+O(L^2/N)$. Subtracting \eqref{eq:thirdorder}, where
$c_0=c_{\mathrm P}-\log2$ and the constant in the $1/L$ term is $\frac18$, gives
\eqref{eq:boundaryshift}. Finally
$J=-\frac12\log u=\frac12(L-\log z)$, and
$\log z=\log L-(\frac12\log L-c_{\mathrm P})/L+O((\log L)^2/L^2)$ gives
\eqref{eq:isingtemperature}.
\end{proof}

\subsection{The crossing window}\label{sec:window}

Before we return to the coupling~\eqref{eq:isingtemperature} in
Section~\ref{sec:coupling}, we describe the whole row near the crossing. The factor $e^{-2y^2}$ in
the next theorem has a simple heuristic. Moving $yN/\sqrt{\log N}$ away from
the middle multiplies the Bessel argument
by $2\sqrt{k(N-k)}/N\approx1-2y^2/\log N$. The argument is about $\log N$, so
it drops by about $2y^2$, and the ratio, which grows like the exponential of the argument,
drops by about the factor $e^{-2y^2}$.

\begin{theorem}[crossing window in the magnetization]\label{thm:critical-window}
Let $X\in\{\mathrm F,\mathrm P\}$ and let $u_N^X$ be the exact central
crossing in the odds under boundary $X$ (so $u_N^{\mathrm F}=u_N$). For fixed real $t,y$ set
\[
u=u_N^X+\frac tN,\qquad
k_N(y)=\operatorname{nint}\!\left(\frac N2+\frac{yN}{\sqrt{\log N}}\right),
\]
where either nearest-integer convention may be used. Then, as $N\to\infty$,
uniformly on compact sets of $(t,y)$,
\begin{equation}\label{eq:criticalprofile}
S^X_{N,k_N(y)}(u_N^X+t/N)\longrightarrow e^{t-2y^2}.
\end{equation}
If $M_N^X(t)$ is the number of interior bins with probability strictly above
the endpoint probability at the odds $u_N^X+t/N$, then
\begin{equation}\label{eq:criticalcount}
\frac{\sqrt{\log N}}N M_N^X(t)\longrightarrow\sqrt{2\max(t,0)}.
\end{equation}
\end{theorem}

So the crossing window has width about $1/N$ in $u$ and about
$N/\sqrt{\log N}$ bins in the magnetization, and inside it the ratio of each bin to an endpoint is a
Gaussian bump times $e^t$. The fixed-$\tau$ limit laws of~\cite{antal2004,herrmann2010} do not
alone provide the uniform local control needed here, since here $\tau$ grows like $\log N$. The
Gaussian itself is expected. At $m=2y/\sqrt{\log N}$ the factor $e^{-2y^2}$ is
$e^{-m^2\log N/2}$, the large-$\tau$ Gaussian $e^{-m^2\tau/2}$ of Antal, Droz
and R\'acz~\cite[Section~4.5, Eq.~(51)]{antal2004} at $\tau\approx\log N$.
Expanding the form of the periodic scaling function of Dantchev, Tonchev and
Rudnick for large $x_t=2Nu$~\cite[Eq.~(4.14)]{dantchev2023} to first order in
$m^2$ gives the same factor.
Theorem~\ref{thm:critical-window} adds the comparison with the endpoint,
centered at the exact crossing and uniform on compact sets.

\begin{proof}
\emph{Step 1. The profile.} Put $\lambda=k/N-\frac12$ and
$\eta=2\sqrt{k(N-k)}/N=\sqrt{1-4\lambda^2}$. For the stated rounding and
bounded $y$,
\[
\eta=1-\frac{2y^2}{L}+O\Bigl(L^{-2}+\frac1{N\sqrt L}\Bigr),
\qquad Nu_N^X=L+O(\log L).
\]
So the Bessel argument $Nu\eta$ at $(u,k)$ is $Nu_N^X+t-2y^2+o(1)$, uniformly
on the compact set. By Theorems~\ref{thm:allbin-bessel}
and~\ref{thm:periodic-bessel} both ratios are within relative error
$O(L^2/N)$ of their Bessel forms, and $I_0(z)$ and $I_1(z)$ equal
$e^z(2\pi z)^{-1/2}(1+O(1/z))$~\cite[Eq.~10.40.1]{dlmf}, so the prefactors
relative to their central values tend to $1$ uniformly. The change of argument
gives $e^{t-2y^2}$, and the exact central ratio is $1$, which proves
\eqref{eq:criticalprofile}.

\emph{Step 2. The count.} The interior ratios increase toward the middle, by
Proposition~\ref{prop:allbin-order} for $X=\mathrm F$ and by the coefficient
comparison after \eqref{eq:periodicratio} for $X=\mathrm P$. If
$t\le0$, the central ratio is at most $1$, so no interior bin is above the
endpoint. Let $t>0$ and $0<\theta<\sqrt{t/2}$. By Step~1, for large $N$ the
bins at scaled distance less than $\sqrt{t/2}-\theta$ are above the endpoint
and those at distance $\sqrt{t/2}+\theta$ are below, and by monotonicity so
are all farther bins. Count the bins in between, divide by $N/\sqrt L$ and let
$\theta\downarrow0$. Rounding changes the count by $O(1)$.
\end{proof}

\begin{remark}[constrained free energies]\label{rem:free-energy}
Let $Z^X_N(\mu)$ be the sum of the Ising weights of the configurations with
magnetization $\mu$ under boundary $X$, and $\mathcal F^X_N(\mu)=-\log Z^X_N(\mu)$ its
dimensionless free energy, as in the fixed-magnetization ensemble
of~\cite{dantchev2022,dantchev2024}. The normalization cancels in the
ratio of a bin to an endpoint, so $-\log S^X_{N,k}=\mathcal F^X_N(2k-N)-\mathcal F^X_N(N)$. On
a compact set of $(t,y)$ the limit $e^{t-2y^2}$ in
\eqref{eq:criticalprofile} is bounded below by a positive constant, so the
convergence there also holds for logarithms, uniformly. Hence
\[
\mathcal F^X_N(2k_N(y)-N)-\mathcal F^X_N(N)\longrightarrow2y^2-t
\]
at the odds $u_N^X+t/N$. Thus, near the crossing the free energy of a
fixed magnetization, measured from full magnetization, is a parabola in the
scaled magnetization, lowered by $t$. This compares single bins. About $N/\sqrt{\log N}$ bins of the window are
each comparable with an endpoint, and the probabilities sum to one, so each
of them has probability at most of order $\sqrt{\log N}/N$ there. The limit
therefore says nothing about the total weight of the nearly ordered
configurations.
\end{remark}

Figure~\ref{fig:window} shows the exact ratios and counts for $N=10^4$,
$10^6$ and $10^8$.

\begin{figure}[t]
\centering
\includegraphics{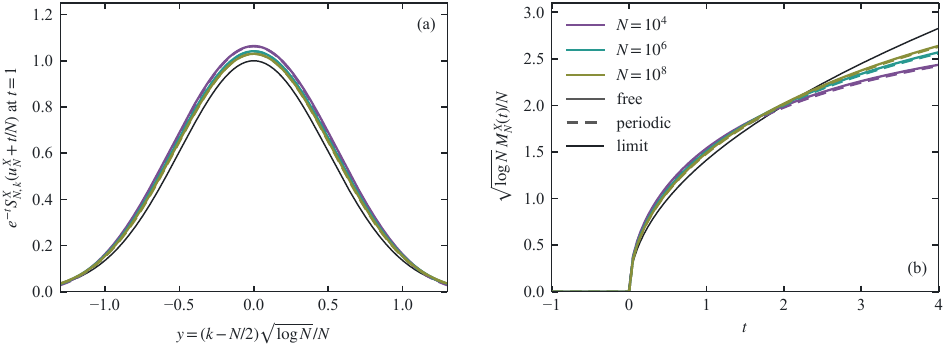}
\caption{(a) The profile $e^{-t}S^X_{N,k}(u^X_N+t/N)$ at $t=1$ against
$y=(k-N/2)\sqrt{\log N}/N$, for free (solid) and periodic (dashed)
ends and $N=10^4$, $10^6$ and $10^8$, with the limit $e^{-2y^2}$ of
Theorem~\ref{thm:critical-window} (thin black line). (b) The scaled number
$\sqrt{\log N}\,M^X_N(t)/N$ of interior bins above the endpoint height, against
$t$, with the limit $\sqrt{2\max(t,0)}$. All curves come from the exact run
polynomials at the exact central crossings. The convergence is slow. At $t=2$
and $y=0$ the relative error of the profile is $0.12$ at $N=10^4$, $0.08$ at
$N=10^6$ and $0.06$ at $N=10^8$, under either boundary.}
\label{fig:window}
\end{figure}

\subsection{The coupling at the crossing}\label{sec:coupling}

We now turn to the coupling at the two central crossings, since this is how
Stepanyan et~al.\@~\cite{stepanyan2024} state the crossing. Recall the free coupling $J_N^{\mathrm F}=-\frac12\log u_N$ and the correlation length $\xi(J)=-1/\log\tanh J$ of
the infinite chain (Section~\ref{sec:setting}), so that $\langle\sigma_0\sigma_r\rangle=e^{-r/\xi(J)}$
there, evaluated at the coupling of the crossing. The spin
correlations in a chain of fixed magnetization behave differently, and can
even become negative~\cite[Sections~II and~IV]{balog2025}.

\begin{corollary}[couplings and correlation lengths at the crossings]\label{cor:couplings}
As $N\to\infty$,
\begin{enumerate}[(a)]
\item The free coupling satisfies $J_N^{\mathrm F}=\frac12(\log N-\log\log N)
+\dfrac{\frac14\log\log N-\frac12c_0}{\log N}
+O\Bigl(\dfrac{(\log\log N)^2}{(\log N)^2}\Bigr)$.
\item The two couplings differ by $J_N^{\mathrm F}-J_N^{\mathrm P}=\dfrac{\log2}{2\log N}
+O\Bigl(\dfrac{\log\log N}{(\log N)^2}\Bigr)$.
\item The correlation lengths satisfy $N/\xi(J_N^{\mathrm F})=2\log N-\log\log N+\log(\pi/8)
+O(\log\log N/\log N)$, and
$N/\xi(J_N^{\mathrm P})=2\log N-\log\log N+\log(\pi/2)
+O(\log\log N/\log N)$.
\end{enumerate}
\end{corollary}

All three parts are facts about $Nu$ at the crossings, read on other scales.
The coupling is $-\frac12\log u$, so it sees $Nu$ only through $\log(Nu)$. The
shift of $\log2$ in $Nu$ between the two boundary conditions therefore becomes
a shift of order $1/\log N$ in the coupling. The correlation length is
$1/(2u)$ to leading order, so $N/\xi$ is about $2Nu$.

\begin{proof}
(a) By \eqref{eq:zN}, $u_N=z_N/(N-z_N)$, so
$J_N^{\mathrm F}=\frac12(L-\log z_N)+O(L/N)$. Taking logarithms in
\eqref{eq:thirdorder} gives
$\log z_N=\log L-(\frac12\log L-c_0)/L+O((\log L)^2/L^2)$, and (a) follows in
the same way as \eqref{eq:isingtemperature} in Step~3 of the proof of
Theorem~\ref{thm:periodic-crossing}.

(b) The difference is $J_N^{\mathrm F}-J_N^{\mathrm P}
=\frac12\log(u_N^{\mathrm P}/u_N^{\mathrm F})=\frac12\log(1+\delta)$, where
$\delta=N(u_N^{\mathrm P}-u_N^{\mathrm F})/(Nu_N^{\mathrm F})$. By
\eqref{eq:boundaryshift} the numerator is $\log2+O(1/L)$. By \eqref{eq:zN} and
Corollary~\ref{cor:every-n}(b) the denominator is $z_N+O(L^2/N)=L+O(\log L)$.
So $\delta=(\log2)/L+O(\log L/L^2)$, and
$\frac12\log(1+\delta)=\delta/2+O(\delta^2)$.

(c) With $u=e^{-2J}$ we have $\tanh J=(1-u)/(1+u)$, so
$N/\xi(J)=N\log\frac{1+u}{1-u}=2Nu+O(Nu^3)$ for $0<u\le\frac12$. At either
crossing $Nu^3=O(L^3/N^2)$. It remains to insert
$Nu_N=z_N+O(L^2/N)=L-\frac12\log L+c_0+O(\log L/L)$, from
Corollary~\ref{cor:every-n}(b), and the expansion \eqref{eq:periodicasym} of
$Nu_N^{\mathrm P}$, and to recall that $2c_0=\log(\pi/8)$ and
$2c_{\mathrm P}=\log(\pi/2)$.
\end{proof}

So at either crossing the correlation length is about $N/(2\log N)$. The chain
is ordered over long stretches, but each stretch is a vanishing fraction of the
chain. In the fixed-$\tau$ regime of~\cite{antal2004}, $N/\xi\approx2\tau$
stays bounded, while here it grows like $2\log N$. By part~(b), the boundary
condition moves the coupling at the crossing by an amount that decays only
like $1/\log N$.

Stepanyan et~al.\@~\cite[Section~III.B]{stepanyan2024} call the periodic
central crossing the second thermal transition of the finite chain. They
give the approximate formula $J\approx(2\log N-1)/5$~\cite[Eq.~(16)]{stepanyan2024},
and their figure~1(b) shows the fitted line $J\approx0.3864\log N-0.21055$.
Table~\ref{tab:couplings} and Figure~\ref{fig:coupling} compare both with the
exact crossing and with \eqref{eq:isingtemperature} without its error term. The line is accurate for
$10\le N\le100$, but neither approximation has a $\log\log N$ term. The main
terms of \eqref{eq:isingtemperature} have slope $\frac12-1/(2\log N)$ in
$\log N$, up to $O(\log\log N/(\log N)^2)$, so the exact curve bends slowly
toward slope $\frac12$. The secant slope of the exact periodic coupling is
$0.385$ on $10\le N\le10^2$, close
to the fitted $0.3864$, then $0.417$ on $10^2\le N\le10^3$, and $0.450$ on
$10^4\le N\le10^5$. At the tabulated sizes $N=10^3,10^4,10^5,10^6$, the main terms of
\eqref{eq:isingtemperature} are closer to the exact value than either
approximation, and at $N=10^5$ they are off by $0.002$, against $0.171$ and
$0.338$.

\begin{table}[t]
\centering
\small
\begin{tabular}{rcccccc}
\toprule
$N$ & $J_N^{\mathrm P}$ & \eqref{eq:isingtemperature} & \cite[(16)]{stepanyan2024} &
\cite[figure~1(b)]{stepanyan2024} & $J_N^{\mathrm F}$ & $J_N^{\mathrm F}-J_N^{\mathrm P}$\\
\midrule
$10$ & $0.687$ & $0.776$ & $0.721$ & $0.679$ & $0.823$ & $0.136$\\
$10^2$ & $1.574$ & $1.597$ & $1.642$ & $1.569$ & $1.658$ & $0.084$\\
$10^3$ & $2.534$ & $2.541$ & $2.563$ & $2.459$ & $2.591$ & $0.057$\\
$10^4$ & $3.540$ & $3.543$ & $3.484$ & $3.348$ & $3.582$ & $0.042$\\
$10^5$ & $4.576$ & $4.578$ & $4.405$ & $4.238$ & $4.610$ & $0.033$\\
$10^6$ & $5.633$ & $5.634$ & $5.326$ & $5.128$ & $5.661$ & $0.028$\\
\bottomrule
\end{tabular}
\caption{The couplings $J_N^{\mathrm P}$ and $J_N^{\mathrm F}$ at the
periodic and free central crossings, from the exact run polynomials and again
from the transfer matrix (the two agree to $10^{-12}$), against
the main terms of \eqref{eq:isingtemperature} and the two approximations of
Stepanyan et~al.~\cite{stepanyan2024} (their Eq.~(16) and the line of their
figure~1(b)). By Corollary~\ref{cor:couplings}(b) the last column is
$(\log2)/(2\log N)$ to leading order, which is $0.030$ at $N=10^5$.
Differences are computed before rounding.}
\label{tab:couplings}
\end{table}

\begin{figure}[htb]
\centering
\includegraphics[width=\textwidth]{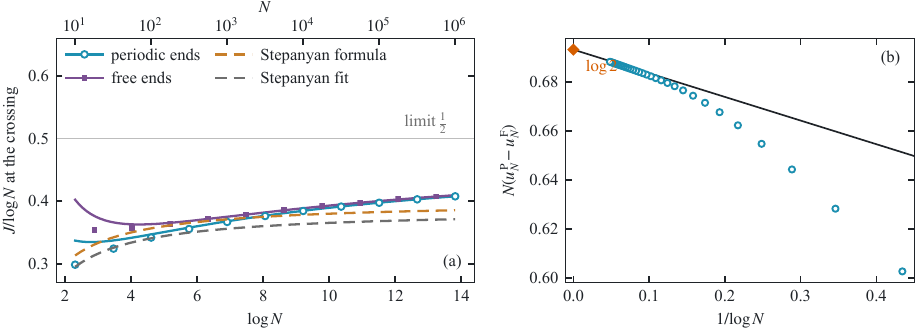}
\caption{(a) The coupling at the central crossing divided by $\log N$,
against $\log N$ for $10\le N\le10^6$, from the exact run polynomials for periodic (circles) and
free (squares) ends, with the main terms of \eqref{eq:isingtemperature} and
of Corollary~\ref{cor:couplings}(a) (solid lines of the same color) and the
two approximations of Stepanyan et~al.~\cite{stepanyan2024} (dashed). The
thin line is the limit $\frac12$. The two approximations have no
$\log\log N$ term and a slope in $\log N$ below $\frac12$, so they drift
away. (b) The boundary shift $N(u^{\mathrm P}_N-u^{\mathrm F}_N)$ for
$10\le N\le10^9$, from the exact run polynomials, against $1/\log N$, with
the first two terms $\log2+(\frac14-\frac12\log2)/\log N$ of
\eqref{eq:boundaryshift} (line). The shift tends to $\log2$ (diamond), and at
$N=10^9$ it differs from the two terms by $3\times10^{-4}$.}
\label{fig:coupling}
\end{figure}

\begin{remark}[other comparisons]\label{rem:convention}
We compare one central bin with one endpoint, as Stepanyan
et~al.\@~\cite{stepanyan2024} do. Comparing instead with both endpoints, or
more generally solving $S^X_{N,m_N}(u)=\Theta$ for a fixed $\Theta>0$, only adds $\log\Theta$
to the constant $c_X$. For periodic ends, Step~2 of the proof of
Theorem~\ref{thm:periodic-crossing} goes through with $\log\Theta$ added to the
right side of its crossing equation. For free ends the same argument works
with Theorem~\ref{thm:bessel}(i) and \eqref{eq:bracket-log} in place of
Theorem~\ref{thm:periodic-bessel}, and with $\frac18$ in place of $\frac38$.
So \eqref{eq:periodicasym} and \eqref{eq:thirdorder} hold for these roots with
$c_X+\log\Theta$ in place of $c_X$, and the coupling moves by $-(\log\Theta)/(2\log N)$
to leading order. For $\Theta=2$ the free constant becomes $c_0+\log2=c_{\mathrm P}$,
so this $\log2$, which comes from the convention, should not be confused
with the shift by $\log2$ between the boundary conditions in
Theorem~\ref{thm:periodic-crossing}. For odd $N$, adding the two central bins
as well as the two endpoints gives $\Theta=1$ again. The bracket of
Theorem~\ref{thm:every-n} is proved only for $\Theta=1$.
\end{remark}

\subsection{Domain walls and the energy gap}\label{sec:walls}

We now describe the crossing in terms of domain walls. Given that
the free chain is in bin $m_N$, a configuration with $j$ domain walls has probability proportional to
$u^{j}$, so the number $j$ of domain walls has the law $W(N,m_N,j)u^j/S_N(u)$. Let $\mathcal W_N(u)=uS_N'(u)/S_N(u)$ be its mean, and put
$\chi(w)=I_0(w)/(I_0(w)+I_1(w))$.

\begin{proposition}[domain walls at the crossing]\label{prop:walls}
Let $N\ge2$ and $u>0$, put $w=Nu$, and suppose that $w(w+2)<2N$. Let
$\Delta=-\log(1-w(w+2)/(2N))$. Then for every $\delta>0$,
\[
\bigl|\mathcal W_N(u)-w-\chi(w)\bigr|\le w(e^\delta-1)+\frac\Delta\delta .
\]
At the crossing, as $N\to\infty$,
$\mathcal W_N(u_N)=z_N+\frac12+\frac1{8z_N}+O\bigl((\log N)^{-2}\bigr)$.
\end{proposition}

So the central bin at the crossing carries about half a domain wall more
than the switching scale $z_N$. Numerically, $\mathcal W_N(u_N)-z_N$ is $0.5191$,
$0.5166$ and $0.5132$ at $N=10^3,10^4,10^5$, against $0.5224$, $0.5162$ and
$0.5126$ for $\frac12+\frac1{8z_N}$.

\begin{proof}
Put $\vartheta=\log u$, $x=hu=w/2$ with $h=N/2$, $g(\vartheta)=\log S_N(e^\vartheta)$ and
$G(\vartheta)=\log F(he^\vartheta)$. Both are logarithms of power series in $e^\vartheta$ with
nonnegative coefficients, so both are convex in $\vartheta$, and $g'(\vartheta)=\mathcal W_N(u)$.
By Theorem~\ref{thm:bessel}(i), $G(\vartheta)-\log h-\Delta\le g(\vartheta)\le G(\vartheta)-\log h$,
and the upper bound holds at every $\vartheta$. From $F'=2F+2I_0(2x)$ we get
$G'(\vartheta)=xF'(x)/F(x)=w+\chi(w)$. The function $\chi$ decreases, since
$I_1/I_0$ is the mean of $\cos t$ under the density proportional to
$e^{w\cos t}$ on $[0,\pi]$~\cite[Eq.~10.32.3]{dlmf}, and the derivative of $I_1/I_0$ in $w$
is the variance of $\cos t$, which is positive. By convexity,
\[
\mathcal W_N(u)\le\frac{g(\vartheta+\delta)-g(\vartheta)}\delta
\le\frac{G(\vartheta+\delta)-G(\vartheta)+\Delta}\delta\le G'(\vartheta+\delta)+\frac\Delta\delta
\le we^\delta+\chi(w)+\frac\Delta\delta,
\]
and in the same way $\mathcal W_N(u)\ge we^{-\delta}+\chi(w)-\Delta/\delta$.
Since $1-e^{-\delta}\le e^\delta-1$, this proves the bound.

At the crossing $w=Nu_N=z_N+O(L^2/N)$ and $\Delta=O(L^2/N)$, and
$\delta=(L/N)^{1/2}$ makes the bound $O(L^{3/2}N^{-1/2})$. The large-argument
forms of $I_0$ and $I_1$ give $\chi(w)=\frac12+\frac1{8w}+O(w^{-2})$, which
proves the second claim.
\end{proof}

Since $u=e^{-2J}$, the definition of $\mathcal W_N$ gives
\[
\frac d{dJ}\log\frac{P_N(m_N)}{P_N(0)}=-2\mathcal W_N(u).
\]
The energy
$-\sum_{i<N}\sigma_i\sigma_{i+1}$ of a configuration exceeds that of the all-up
configuration by twice its number of domain walls. So at the crossing the mean
energy gap between the central bin and the all-up configuration is
$2\mathcal W_N(u_N)=2\log N-\log\log N+O(1)$, by Proposition~\ref{prop:walls} and Corollary~\ref{cor:every-n}(b).
Over a range of couplings of width
$1/(2\mathcal W_N(u_N))$, which is $0.061$ at $N=10^4$ and $0.048$ at $N=10^5$,
the ratio of the central bin to an endpoint changes by about a factor $e$. The
same proof with Theorem~\ref{thm:periodic-bessel} in place of
Theorem~\ref{thm:bessel}, where $G$ becomes $\log(xI_1(2x))$ and
$G'=wI_0(w)/I_1(w)$, gives the mean number $Nu_N^{\mathrm P}+\frac12+O(1/\log N)$ of domain walls
in the central bin at the periodic crossing, for even $N$.

The count of about $\log N$ walls also accounts for a second fit of
Stepanyan et~al. They measure the mean
domain length $l_d=\langle N/(n_d+\delta_{0,n_d})\rangle$, with $n_d$ the
number of domain walls of the periodic chain, at their second transition and
fit $l_d\propto N^{0.757}$ on $20\le N\le100$~\cite[Eq.~(18)]{stepanyan2024}.
At the coupling $J_N^{\mathrm P}$ of the exact periodic crossing, where in the
whole chain $n_d=2r$ has weight proportional to
$\binom N{2r}u^{2r}$, we get $l_d=9.27$ at $N=20$ and $l_d=31.21$ at $N=100$,
and a least-squares exponent $0.755$ over the even $N$ in that range. With about $\log N$ walls
one expects $l_d$ of order $N/\log N$, whose least-squares exponent on that
range is $0.74$ and tends to $1$ only slowly. So the fitted power is close to a local slope of $N/\log N$, just as the slope
$0.3864$ of their fitted coupling is close to the secant slope $0.385$ of the
exact coupling on $10\le N\le10^2$.

\begin{remark}[full order against one bin]\label{rem:ordered}
The chain is fully ordered, all up or all down, with probability
$P_N(0)+P_N(N)=p^{N-1}$. This is $\frac12$ at $u_{\mathrm{ord}}=2^{1/(N-1)}-1$, that
is, at the coupling
$J_{\mathrm{ord}}=\frac12\log\frac{N-1}{\log2}+O(1/N)$. By
\eqref{eq:zN}, $J_N^{\mathrm F}=\frac12\log(N/z_N-1)$, so
\[
J_{\mathrm{ord}}-J_N^{\mathrm F}=\tfrac12\log\frac{z_N}{\log2}+O\Bigl(\frac{\log N}N\Bigr)
=\tfrac12\log\log N-\tfrac12\log\log2+o(1).
\]
The difference is $1.046$ at $N=10^3$ and $1.330$ at $N=10^5$. So full order
wins against the rest of the chain only at a coupling about $\frac12\log\log N$
above the central crossing. The $\log\log N$ in $J_N^{\mathrm F}$ measures how
much less likely one central bin is than the whole disordered band.
\end{remark}

\subsection{A magnetic field}\label{sec:field}

A dimensionless field $\mathcal B=\beta H>0$ ($\beta h$ in the notation
of~\cite{stepanyan2024}) multiplies the weight of a
configuration by $e^{\mathcal B(\sigma_1+\dots+\sigma_N)}$. Magnetization
$N-2a$, with $a$ spins down, then has weight $e^{-2\mathcal Ba}$ relative to
full magnetization, beyond the zero-field ratio. By the symmetry of the row, its
ratio to the all-up configuration is
\[
R^X_{N,a}(u)=S^X_{N,a}(u)\,e^{-2\mathcal Ba}\qquad(1\le a\le N-1),
\]
for either boundary $X\in\{\mathrm F,\mathrm P\}$. Garc\'{\i}a-Pelayo~\cite[Eq.~(17)]{garciapelayo2009}
gives this reweighting for both boundary conditions. We keep $\mathcal B$ fixed
as $N$ grows and as the coupling varies (this is different from varying the
temperature at fixed physical field $H$). The ratio increases strictly from $0$ to $\infty$ in $u$, so
it equals $1$ at a single $u=u^{(a)}_N$, and the all-up configuration is more
likely than magnetization $N-2a$ exactly when $u<u^{(a)}_N$. The roots depend
on $X$ and $\mathcal B$, which we suppress in the notation.

There are two comparisons behind the transitions of Stepanyan
et~al.\@~\cite[Section~III.C]{stepanyan2024}. The one-spin crossing is
$u^{(1)}_N$. For $N\ge3$, define
\[
u^{(\ge2)}_N=\min_{2\le a\le N-1}u^{(a)}_N.
\]
At this value the all-up configuration ties with the most likely bin having
at least two down spins, excluding the all-down configuration (whose ratio
to all up is always $e^{-2\mathcal BN}<1$). The all-up configuration is
strictly more likely than every other magnetization bin exactly when
$u<\min\{u^{(1)}_N,u^{(\ge2)}_N\}$. Even when there is only one transition in the sense of that paper, these two
crossing values can differ.

Both comparisons come from the edge of the row. The right variable is
$y=Nu^2$. Then $bu^2=(1-a/N)y$ for the variable $bu^2$ of
Section~\ref{sec:allbin-edge}, and this factor tends to one for fixed $a$. With $H_a$ of
\eqref{eq:edge-polynomial}, let $y_a(\mathcal B)$ be the positive root of $H_a(y)=e^{2\mathcal Ba}$. So
$y_1(\mathcal B)=e^{2\mathcal B}$, and at $\mathcal B=0$ these are the roots
$y_a$ of Section~\ref{sec:allbin-edge}.

The weights $e^{-2\mathcal Ba}H_a(y)$ are classical. Put $\lambda=e^{-2\mathcal B}$
and $\nu=y\lambda/(1-\lambda)$. Then $e^{-\nu}$ and $e^{-\nu}\lambda^aH_a(y)$ for
$a\ge1$ are the probabilities of the P\'olya--Aeppli law, with Anscombe's $\tau$
equal to $\lambda$ and his $m_1$ equal to $\nu$~\cite[Eqs.~(2.11) to~(2.13),
p.~362]{anscombe1950}. Anscombe remarks that this law has two modes when
$2<\nu<(1-\lambda)^{-1}$~\cite[p.~364]{anscombe1950}, and this range is
nonempty exactly when $\mathcal B<\frac12\log2$. His remark is about local
modes of the limiting law. Theorem~\ref{thm:field} is about the finite
chain, and it controls the crossings and the minimum over all competing
bins for both boundary conditions.

\Needspace{24\baselineskip}
\begin{theorem}[edge crossings in a fixed field]\label{thm:field}
Fix a dimensionless field $\mathcal B>0$ and either free or periodic
ends. Let $y_a(\mathcal B)>0$ solve $H_a(y)=e^{2\mathcal Ba}$, with $H_a$
as in \eqref{eq:edge-polynomial}.
\begin{enumerate}[(a)]
\item For every fixed integer $a\ge1$,
\[
N(u^{(a)}_N)^2\longrightarrow y_a(\mathcal B)
\qquad(N\to\infty).
\]
In particular, $N(u^{(1)}_N)^2\to e^{2\mathcal B}$.
\item As $a\to\infty$, $y_a(\mathcal B)\to\infty$, so the minimum
\[
y_*(\mathcal B)=\min_{a\ge2}y_a(\mathcal B)
\]
is attained, and
\[
N(u^{(\ge2)}_N)^2\longrightarrow y_*(\mathcal B)
\qquad(N\to\infty).
\]
\item With $\mathcal B_c=\frac12\log2$,
\[
\begin{cases}
y_*(\mathcal B)<e^{2\mathcal B},&0<\mathcal B<\mathcal B_c,\\
y_*(\mathcal B)=e^{2\mathcal B}=2,&\mathcal B=\mathcal B_c,\\
y_*(\mathcal B)>e^{2\mathcal B},&\mathcal B>\mathcal B_c.
\end{cases}
\]
\end{enumerate}
\end{theorem}

In terms of the coupling, these two crossings are
\[
J_N^{(1)}=\tfrac14\log N-\tfrac12\mathcal B+o(1),\qquad
J_N^{(\ge2)}=\tfrac14\log N-\tfrac14\log y_*(\mathcal B)+o(1).
\]
The field moves the last crossing with all up to the edge scale
$J=\frac14\log N+O(1)$ (in zero field it comes from the middle, with
leading coupling $\frac12\log N$). For large $N$, if
$\mathcal B<\mathcal B_c$, another bin is still more likely when the
one-spin bin ties with all up. So all up becomes most likely at a later
crossing, separated from the one-spin crossing by a positive limiting
gap in $J$. If $\mathcal B>\mathcal B_c$, the one-spin bin is last to
fall below all up. This gives only one transition to an all-up maximum,
though the two crossing values defined above are still different.
At $\mathcal B=\mathcal B_c$, their coupling gap tends to zero as
$N\to\infty$ (the two crossings need not agree at finite $N$).

This recovers the limiting dimensionless threshold in Stepanyan
et~al.\@~\cite[Eq.~(20)]{stepanyan2024}. For large $N$ their physical threshold is
$h_c\approx2J_{\mathrm{ex}}\log2/\log(N/2)$, where $J_{\mathrm{ex}}$ denotes their
exchange coupling. At the merging point
$J=\beta J_{\mathrm{ex}}\sim\frac14\log N$, so $\beta h_c\to\frac12\log2$.
Theorem~\ref{thm:field} gives the same limit for their periodic chain, and
also for free ends. The agreement
comes from the common fixed-edge limit $H_a$ proved below.

\begin{proof}
(a) For free ends we use \eqref{eq:edge-exact-series}.
With $b=N-a$ and $u=\sqrt{y/N}$, the edge variable $bu^2=y(1-a/N)$ tends to $y$,
$V_r=U_r(b)\to1$ for each $r\le a$, and the first two sums carry the factors
$2\sqrt{y/N}$ and $y/N$, so $S^{\mathrm F}_{N,a}(u)\to H_a(y)$ uniformly for $y$
in compact sets. For periodic ends, in \eqref{eq:periodicratio} the term with
$r$ runs is $\frac Nr\binom{a-1}{r-1}\binom{b-1}{r-1}u^{2r}$, and
$\binom{b-1}{r-1}u^{2r-2}\to y^{r-1}/(r-1)!$ uniformly in the same sense, so
$S^{\mathrm P}_{N,a}(u)\to\sum_{r=1}^a\binom{a-1}{r-1}y^r/r!=H_a(y)$. The
limit $H_a(y)e^{-2\mathcal Ba}$ increases strictly and continuously in $y$ from
$0$ to $\infty$, and $R^X_{N,a}$ increases in $u$, so the roots converge.

(b) We need a bound uniform in $a$. Let $Y>0$ and $0<y\le Y$. By
Theorems~\ref{thm:allbin-bessel} and~\ref{thm:periodic-bessel} the ratio is at
most its Bessel form, and $I_0,I_1\le e^z$. For free ends with
$s=\sqrt{a(N-a)}$ and $c=N/(2s)$ this gives
$S^{\mathrm F}_{N,a}(u)\le2u(1+c)e^{2su}$, and for periodic ends
$S^{\mathrm P}_{N,a}(u)\le(Nu/s)e^{2su}$. For every interior bin,
$a(N-a)\ge N-1$, so $Nu/s\le\sqrt{2y}$. Also
$2su\le2\sqrt{ya}$ and $2u\le2\sqrt y$. Thus for either boundary and every
$1\le a\le N-1$,
\[
R^X_{N,a}(u)\le4\sqrt Y\,e^{2\sqrt{Ya}-2\mathcal Ba},
\]
which is below $1$ for every $a\ge a_Y$, uniformly in $N$, for some $a_Y$
that depends only on $Y$ and $\mathcal B$. In particular,
every such interior bin has $N(u^{(a)}_N)^2>Y$. Letting
$N\to\infty$ in the displayed bound with $a$ fixed, and using (a), gives
$H_a(y)e^{-2\mathcal Ba}\le4\sqrt Y\,e^{2\sqrt{Ya}-2\mathcal Ba}<1$ for
$a\ge a_Y$ and $y\le Y$, so $y_a(\mathcal B)>Y$. Hence $y_a\to\infty$ and
the minimum is attained. Take $Y=y_*(\mathcal B)+1$. For large $N$ the minimum
over $a\ge2$ of $N(u^{(a)}_N)^2$ is a minimum over $2\le a<a_Y$, and (a) gives
its limit.

(c) Put $\lambda=e^{-2\mathcal B}$ and $y=e^{2\mathcal B}=1/\lambda$. Then
\[
H_a(y)e^{-2\mathcal Ba}=\sum_{r=0}^{a-1}\binom{a-1}r\frac{\lambda^{a-1-r}}{(r+1)!}
\le\sum_{r=0}^{a-1}\binom{a-1}r\lambda^{a-1-r}2^{-r}=\Bigl(\lambda+\frac12\Bigr)^{a-1},
\]
since $(r+1)!\ge2^r$. If $\mathcal B>\frac12\log2$, then $\lambda<\frac12$ and the
right side is below $1$ for every $a\ge2$. As $H_a$ increases, this gives
$y_a(\mathcal B)>e^{2\mathcal B}$ for every $a\ge2$, and so
$y_*(\mathcal B)>e^{2\mathcal B}$ because the minimum is attained. If
$\mathcal B<\frac12\log2$, then at $a=2$ the sum is $\lambda+\frac12>1$, so
$y_2(\mathcal B)<e^{2\mathcal B}$. At $\mathcal B=\frac12\log2$, the same
bound gives $H_a(2)e^{-2\mathcal Ba}\le1$ for every $a\ge2$, with equality
at $a=2$. Hence $y_a(\mathcal B)\ge2$ and $y_*(\mathcal B)=y_2(\mathcal B)=2$.
\end{proof}

The minimizing $a$ grows as the field weakens in the following examples.
For $\mathcal B=0.1$, $0.2$, $0.3$ and $0.5$, respectively, it is
$11$, $4$, $2$ and $2$, and $y_*=0.757$, $1.287$, $1.764$ and $2.972$.
At $N=10^4$, in the same order of fields,
$N(u_N^{(\ge2)})^2$ is $0.758$, $1.288$, $1.764$ and $2.973$ for periodic
ends, and $0.741$, $1.266$, $1.738$ and $2.939$ for free ends. We checked
the values at $N=10^4$ directly from the full run polynomials. Interval bounds
give the displayed digits and exclude competing bins near the edge,
and the tail bound in the proof of Theorem~\ref{thm:field}(b) excludes all the
remaining bins. The free
chain converges more slowly because of the extra terms from its ends
(of relative size $N^{-1/2}$). Here the field is fixed as $N$ grows.
If it decreases with $N$, the comparison should move back toward the middle
of the row, and that crossover remains open.

\FloatBarrier

\section{Discussion}\label{sec:discussion}

The crossing studied by Stepanyan et~al.\@~\cite{stepanyan2024} compares
the all-up configuration with one central magnetization bin. Around
the crossing there is a band of magnetizations of width
$N/\sqrt{\log N}$, with the parabolic free-energy profile of
Remark~\ref{rem:free-energy}. For free ends the central bin has about
$\log N$ walls on average, giving the mean energy gap
\[
2\mathcal W_N(u_N)=2\log N-\log\log N+O(1)
\]
above all up, in units of the exchange coupling
(Proposition~\ref{prop:walls}). Thus a factor $e$ change in their
probability ratio takes a coupling change of about
$1/(2\mathcal W_N(u_N))$. For this zero-field central-bin comparison
the width decreases like $1/\log N$, whereas the usual $1/N$ rounding at
a first-order transition assumes a nonzero latent heat per spin, which
gives an energy gap proportional to the
volume~\cite{privman1983,borgs1990,lee1991}. We therefore interpret
this crossing as a finite-size crossover, consistent with the
microcanonical analysis of finite chains in~\cite{sitarachu2020}.
The $\log\log N$ term comes from balancing one wall-free configuration
with the many configurations in a central bin. The $\log2$ boundary
shift comes from the restriction to even numbers of walls on a ring.

There are three relevant coupling scales. At the adjacent crossing,
$\xi\approx\sqrt N/2$ (Proposition~\ref{prop:adjacent}). For periodic
ends this is the first crossing of~\cite{stepanyan2024}, exactly at
$J=\frac14\log N$. At the central crossing,
$\xi\approx N/(2\log N)$ (Corollary~\ref{cor:couplings}). For free
ends, full order (all up or all down) has probability $\frac12$ when
$N/\xi\approx2\log2$ (Remark~\ref{rem:ordered}). Only at this last
scale is the correlation length comparable with the chain length.
Its coupling lies $\frac12\log\log N+O(1)$ above the central crossing.
This separation matters because the central crossing compares all up
with just one bin. Full order must become more likely still before
its weight equals that of the whole remaining chain.

The Bessel bounds behind these statements (Theorems~\ref{thm:bessel},
\ref{thm:allbin-bessel} and~\ref{thm:periodic-bessel}) hold at every $N$ and
every coupling.
Through Corollary~\ref{cor:fixed-M}, Theorem~\ref{thm:allbin-bessel} gives the
fixed-magnetization partition function of the free chain used
in~\cite{dantchev2022,dantchev2024} with an explicit relative error. The error
tends to zero when $Nu^2\to0$, in particular at the crossing, where the mean number of
domain walls grows like $\log N$. The approach extends to
the $d$-state Potts chain~\cite{paperB}. Three questions are open. First, the
bracket of Theorem~\ref{thm:every-n} is proved only for free ends and for the
level $\Theta=1$ of Remark~\ref{rem:convention}, and the same method should give it
for periodic ends and for other levels. Second, a field that shrinks with $N$
should move the comparison from the edge back toward the middle of the row,
and we do not treat it. Third, Theorem~\ref{thm:field} has no rate of convergence.

\section{Verification}\label{sec:verify}

Remark~\ref{rem:sign} is computer-assisted. Otherwise the proofs use two
computations. The first consists of the integer certificates for $2\le N\le174$, in Step~4
of the proof of Theorem~\ref{thm:every-n} and in
Proposition~\ref{prop:crossing-eq}. The second is the evaluation at $N=2774$
in Corollary~\ref{cor:every-n}(d). Both are described where they are used.
Separate Python scripts check the exact law, the orderings, the Bessel
bounds, the asymptotic formulas and the couplings of
Table~\ref{tab:couplings}. Only one of them needs a library outside the
standard one. Where possible a number is computed in two ways. For example,
the couplings of Table~\ref{tab:couplings} come from the run polynomials and
again from the transfer matrix of the chain (using NumPy), and the two agree
to $10^{-12}$.

Every numbered result of the paper is formalized in Lean~4, in the folder
\nolinkurl{kagey-problems/problem131/lean/PaperA} of the repository named in
the Data availability statement, except the Mehler--Heine limit for real
parameters in Remark~\ref{rem:jacobi}, which we cite (Lean proves it for
nonnegative integer parameters). The results stated at every $N$, among them
Theorem~\ref{thm:every-n} and Proposition~\ref{prop:crossing-eq} for all
$N\ge2$, are proved as stated. Each asymptotic error term is replaced by a
bound that holds from some point on and is at least as strong as the paper's,
with explicit constants in Corollary~\ref{cor:every-n}(b) and~(c),
Corollary~\ref{cor:couplings} and Proposition~\ref{prop:walls}. For
$2\le N\le174$ the Lean kernel checks the same certificates as Step~4 of the
proof of Theorem~\ref{thm:every-n}, with rational bounds on $\zeta_N$ that Lean
proves itself, and for Proposition~\ref{prop:crossing-eq} it checks rational
certificates of its own on the same brackets. For
Corollary~\ref{cor:every-n}(d) Lean proves its own rational bounds on
$\zeta_{2774}$ and extends the inequality to every larger $N$, and for
Remark~\ref{rem:sign} it checks one rational certificate between $z_N$ and
$\zeta_N$ for each $2\le N\le2773$. Lean also proves the finite Ising
dictionary and the large-argument forms of $I_0$ and $I_1$ of
Section~\ref{sec:setting}, the correlations
$\langle\sigma_0\sigma_r\rangle=e^{-r/\xi(J)}$ of the infinite chain (with
$\langle\sigma_i\sigma_{i+r}\rangle=\tanh^rJ$ in every finite free chain), the
consequences displayed after Theorem~\ref{thm:allbin-expansion}, and that
$R^X_{N,a}$ is the ratio of the Ising weights in the field of the
magnetizations $N-2a$ and $N$. With kernel-checked rational bounds it checks
that every value printed in Tables~\ref{tab:every-n} and~\ref{tab:couplings} is
the correct rounding of the quantity it names, and it bounds most decimals
quoted in the text and the captions. Among the parts not formalized are the couplings
$J_N^{(1)}$ and $J_N^{(\ge2)}$ displayed after Theorem~\ref{thm:field}, the
mean number of domain walls at the periodic crossing stated after
Proposition~\ref{prop:walls}, the Lambert form of $\zeta_N$, and the numbers
$0.755$, $0.74$ and $0.013$ of Sections~\ref{sec:walls}
and~\ref{sec:allbin-layer}. Only the standard axioms are used. The file
\nolinkurl{MAP.md} in that folder gives the Lean theorem for each result, or
says that there is none.

\section*{Acknowledgments}

I thank Peter Kagey, whose Open Problem
Collection~\cite{kagey131} asks for this crossing for a Galton board whose
ball keeps its direction with probability $p$.

\paragraph{Use of AI tools.} I used Claude Opus 5.5 (Anthropic, through
Claude Code, in September and October 2026) and GPT-6 Astra (OpenAI,
through Codex) in this work. I had conjectured a bracket for the crossing at
every $N$, but it had gaps. Claude Opus 5.5 proposed an improved version, with a
proof that uses the Bessel bounds of Lemma~\ref{lem:bessel-bounds}, and I then
wrote my own version, which is Theorem~\ref{thm:every-n} as stated here. Claude
Opus 5.5 also proposed Corollaries~\ref{cor:fixed-M} and~\ref{cor:couplings}
and Proposition~\ref{prop:walls}, and drafted their proofs, which I rewrote
by hand. The other results and their proofs are my own. I used both
models to check and revise the paper, to search the literature and draft some
of the comparisons with it, and to write parts of the Python verification
scripts and the figures.

The Lean~4 formalization in \nolinkurl{kagey-problems/problem131/lean/PaperA}
was generated using Claude Opus 5.5 through conversation, and it was revised
and checked against the results by me (\nolinkurl{MAP.md} in that folder gives
the Lean theorem for each result). Every proof, in the form present in this
paper, was written by me, and all mathematical claims, proofs and computations
in this paper are my responsibility.

\paragraph{Funding.} This work received no funding.

\paragraph{Conflict of interest.} The author declares no conflict of
interest.

\section*{Data availability statement}

The code and data are in the folder \nolinkurl{kagey-problems/problem131} of
\url{https://github.com/apemm/Kagey-Problems}, in the release
\texttt{paperA-arxiv-v3}\ifx\zenodoDOI\empty\else, with an archived snapshot at
\url{https://doi.org/\zenodoDOI}\fi. Run
\texttt{python -B verify\_all.py -{}-paper a} in that folder to repeat the
checks. \nolinkurl{SUPPLEMENT.md} describes each script, and
\nolinkurl{paperA/COMPUTATION.md} gives the commands and dependencies. The
certificates of Theorem~\ref{thm:every-n} are in
\nolinkurl{paperA/data/every_n_certificates.csv}. The script
\nolinkurl{paperA/verify_ising_coupling.py} prints
Table~\ref{tab:couplings}. Numbers in Sections~\ref{sec:middle}
and~\ref{sec:periodicwindow} that no script prints are in
\nolinkurl{paperA/data/quoted_values.txt}.

\appendix
\setcounter{equation}{0}\renewcommand{\theequation}{A.\arabic{equation}}%
\renewcommand{\theHequation}{A.\arabic{equation}}
\section{Proofs of the bracket and of the crossing equation}\label{app:every-n}

This appendix proves Lemma~\ref{lem:bessel-bounds} (Appendix~\ref{app:bessel-bounds}), the analytic half of
Theorem~\ref{thm:every-n} for $N\ge175$ (Appendix~\ref{app:every-n-large}), the identity behind the
certificates (Appendix~\ref{app:certificates}) and Proposition~\ref{prop:crossing-eq} (Appendix~\ref{app:crossing-eq}). The letters $s$, $v$, $y$, $\lambda_k$,
$g_\pm$, $r$, $r'$, $t$, $t'$, $\tau$,
$\kappa$, $m$, $B$, $D$, $B'$, $D'$, $\delta$ and $\phi$ are local to
this appendix, and each is defined where it first shows up. The letters $Y$, $Y'$, $\bar w$ and
$\bar w'$ have the meaning of the proof of Theorem~\ref{thm:every-n}. The letter $t$
is an angle in Appendix~A.1 and a bound in Appendices~A.2 and~A.4.
We round every decimal constant in the safe direction, namely so that the
inequality it enters stays true.

\subsection{Proof of Lemma~\ref{lem:bessel-bounds}}\label{app:bessel-bounds}

For integer $n$, $I_n(z)=\frac1\pi\int_0^\pi e^{z\cos t}\cos(nt)\,dt$
\cite[Eq.~10.32.3]{dlmf}, so
$I_0(z)+I_1(z)=\frac1\pi\int_0^\pi e^{z\cos t}(1+\cos t)\,dt$. Substitute
$s=1-\cos t$, so that $\sin t=\sqrt{s(2-s)}$ and $1+\cos t=2-s$, and then
$s=v/z$. This gives
\begin{equation}\label{eq:E-integral}
\cE(z)=\frac1{\sqrt\pi}\int_0^{2z}e^{-v}v^{-1/2}\sqrt{1-\frac v{2z}}\,dv .
\end{equation}
On $[0,1]$ we have $\sqrt{1-y}=1-\sum_{k\ge1}\lambda_ky^k$ with every $\lambda_k>0$,
$\lambda_1=\frac12$ and $\lambda_2=\frac18$, and $\sum_k\lambda_k=1$ by monotone convergence as
$y\uparrow1$. So $\sum_{k\ge3}\lambda_k=\frac38$, and $y^k\le y^3$ gives
$\sum_{k\ge3}\lambda_ky^k\le\frac38y^3$. Hence, for $0\le y\le1$,
\[
g_-(y):=1-\frac y2-\frac{y^2}8-\frac{3y^3}8\le\sqrt{1-y}\le1-\frac y2-\frac{y^2}8
\le g_+(y):=1-\frac y2-\frac{y^2}8+\frac{y^3}8 .
\]
The cubic $g_-$ decreases and $g_-(1)=0$, so $g_-\le0$ on $[1,\infty)$. Also
$8g_+(y)=(y-2)^2(y+1)+2(y-1)^2+2>0$ for $y\ge-1$. Put $y=v/(2z)$ in
\eqref{eq:E-integral}, replace $\sqrt{1-y}$ by $g_\pm$ and extend the integral
to $\infty$. The added piece is at most $0$ with $g_-$ and at least $0$ with
$g_+$, so the lower bound survives with $g_-$ and the upper with $g_+$. Since
$\frac1{\sqrt\pi}\int_0^\infty e^{-v}v^{k-1/2}\,dv=\Gamma(k+\frac12)/\Gamma(\frac12)$
equals $1,\frac12,\frac34,\frac{15}8$ for $k=0,1,2,3$, the two bounds integrate
to the two sides of the lemma. The bound $\cE\le1$ follows in the same way from
$\sqrt{1-y}\le1$. Finally $-\frac3{128z^2}+\frac{15}{512z^3}\le0$ exactly when
$z\ge\frac54$.

\subsection{Proof of Theorem~\ref{thm:every-n} for \texorpdfstring{$N\ge175$}{N >= 175}}\label{app:every-n-large}

\emph{Step 1. Bounds that hold for every $\zeta\ge4$.} Fix $N\ge175$ and write
$\zeta=\zeta_N$, $Y=\zeta+\frac1{8\zeta}$, $r=Y/N$ and
$n(\zeta)=\sqrt{8\zeta/\pi}\,e^\zeta$. Then $n(\zeta_N)=N$ by
\eqref{eq:zetaN}, and $\zeta\ge4$ since $n(4)<175$. For $k=1,2,3$ the function
$Y^k/n(\zeta)$ decreases on $[4,\infty)$, because its logarithmic derivative is
$k(1-\frac1{8\zeta^2})/Y-1-\frac1{2\zeta}<\frac kY-1<0$. The same holds for
$Y^4/n$, since $Y>4$, and for $\zeta/n$ and $\zeta^2/n$. The values at
$\zeta=4$ give, for every $\zeta\ge4$,
\[
r\le0.023135,\quad\frac{Y^2}N\le0.093262,\quad\frac{Y^3}N\le0.375960,\quad
\frac{Y^4}N\le1.51559,\quad\frac{\zeta(\zeta+1)}N\le0.114777 .
\]

\emph{Step 2. The upper bound.} Let $\bar w=Y/(1-r)$, so that $N\bar w/(N+\bar w)=Y$, and let
$\delta=\bar w(\bar w+2)/(2N)$. By Lemma~\ref{lem:bracket}(b) and \eqref{eq:bracket-log}
it suffices to show that $\mathcal G:=\psi(\bar w)-\psi(\zeta)+\log \cE(\bar w)+\log(1-\delta)>0$.
We bound the three parts from below. First,
$\psi(\bar w)-\psi(\zeta)=(\bar w-Y)+(Y-\zeta)+\frac12\log\frac Y\zeta+\frac12\log\frac{\bar w}Y$,
where $\bar w-Y=\frac{Yr}{1-r}$, $Y-\zeta=\frac1{8\zeta}$,
$\log\frac Y\zeta=\log(1+\frac1{8\zeta^2})\ge\frac1{8\zeta^2}-\frac1{128\zeta^4}$
and $\log\frac{\bar w}Y=-\log(1-r)\ge r$. Second, Lemma~\ref{lem:bessel-bounds} gives
$\cE(\bar w)\ge1-t$ with $t=\frac1{8\bar w}+\frac3{128\bar w^2}+\frac{45}{512\bar w^3}$. Since
$\bar w\ge Y\ge\zeta$ and
$\frac1{8Y}=\frac1{8\zeta}-\frac1{8\zeta(8\zeta^2+1)}\le\frac1{8\zeta}-\frac1{64\zeta^3}+\frac1{512\zeta^5}$,
we get
$t\le\tau:=\frac1{8\zeta}+\frac3{128\zeta^2}+\frac{37}{512\zeta^3}+\frac1{512\zeta^5}$
and $\log \cE(\bar w)\ge\log(1-\tau)\ge-\tau-\frac{\tau^2}{2(1-\tau)}$. Third,
$\delta=\frac{Yr}{2(1-r)^2}+\frac r{1-r}$ and
$\log(1-\delta)\ge-\delta-\frac{\delta^2}{2(1-\delta)}$. Adding the three parts
gives $\mathcal G\ge B+D$, where $B$ collects the terms in $\zeta$ alone and $D$ the
terms with a factor $r$,
\[
B=\frac5{128\zeta^2}-\frac{37}{512\zeta^3}-\frac1{256\zeta^4}-\frac1{512\zeta^5}-\frac{\tau^2}{2(1-\tau)},\qquad
D=\frac r{2(1-r)}\Bigl(\frac{Y(1-2r)}{1-r}-1-r\Bigr)-\frac{\delta^2}{2(1-\delta)} .
\]
Both $\zeta\tau$ and $\tau$ decrease in $\zeta$, so $\zeta^2B$ increases. At
$\zeta=4$ we have $\zeta\tau<0.135384$, $\tau<0.033846$ and
$\zeta^2B>0.011236$. Hence $B>0$. In $D$ the bracket is at least
$4.03125(1-0.023683)-1.023135>2.912$, so the first term of $D$ is at least
$1.456\,r$. Write $\delta=rm$ with $m=\frac Y{2(1-r)^2}+\frac1{1-r}$. Then
$\delta\le0.07255$ and
$rm^2=\frac{Y^3}{4N(1-r)^4}+\frac{Y^2}{N(1-r)^3}+\frac r{(1-r)^2}\le0.2276$, so
$\frac{\delta^2}{2(1-\delta)}=r\cdot\frac{rm^2}{2(1-\delta)}\le0.1228\,r$.
Hence $D\ge1.33\,r>0$. So $\mathcal G>0$ and $z_N<Y$.

\emph{Step 3. The lower bound.} Let $\varepsilon=\varepsilon_N$,
$Y'=Y-\varepsilon$, $r'=Y'/N$ and $\bar w'=Y'/(1-r')$, so that
$N\bar w'/(N+\bar w')=Y'$. By Step~1, $\varepsilon\le0.114777+\frac1{256}<0.1187$, so
$Y'\ge\zeta-0.1187\ge3.88$ and $\bar w'\ge Y'$. By Lemma~\ref{lem:bracket}(a)
and \eqref{eq:bracket-log} it suffices to show that
$\mathcal G':=\psi(\bar w')-\psi(\zeta)+\log \cE(\bar w')<0$. We bound the two parts from above.
First,
$\psi(\bar w')-\psi(\zeta)=(Y'-\zeta)+\frac{Y' r'}{1-r'}+\frac12\log\frac{Y'}\zeta-\frac12\log(1-r')$,
where $Y'-\zeta=\frac1{8\zeta}-\varepsilon$,
$\frac12\log\frac{Y'}\zeta\le\frac12(\frac{Y'}\zeta-1)=\frac1{16\zeta^2}-\frac\varepsilon{2\zeta}$
and $-\frac12\log(1-r')\le\frac{r'}{2(1-r')}$. The two terms in $r'$ add up to
$\frac{Y'^2+Y'/2}{N(1-r')}$. Second, Lemma~\ref{lem:bessel-bounds} gives
$\cE(\bar w')\le1-t'$ with
$t'=\frac1{8\bar w'}+\frac3{128\bar w'^2}-\frac{15}{512\bar w'^3}\in(0,1)$, so
$\log \cE(\bar w')\le-t'$. Here $\frac1{8\bar w'}=\frac1{8Y'}-\frac1{8N}$. Write
$Y'=\zeta(1+\kappa)$ with $\kappa=\frac1{8\zeta^2}-\frac\varepsilon\zeta$.
Then $\frac1{Y'}\ge\frac{1-\kappa}\zeta$, so
$-\frac1{8\bar w'}\le-\frac1{8\zeta}+\frac1{64\zeta^3}-\frac\varepsilon{8\zeta^2}+\frac1{8N}$.
Next, $x\mapsto x/(1-x/N)$ increases and $Y'\le Y$, so $\bar w'\le Y/(1-r)$ and
$\frac1{\bar w'^2}\ge\frac{(1-r)^2}{\zeta^2(1+\frac1{8\zeta^2})^2}\ge\frac{(1-2r)(1-\frac1{4\zeta^2})}{\zeta^2}$.
Finally $\frac{15}{512\bar w'^3}\le\frac{15}{512Y'^3}$. Adding, and using
$\varepsilon(1+\frac1{2\zeta})=\frac{\zeta^2+\frac32\zeta+\frac12}N+\frac1{16\zeta^2}+\frac1{32\zeta^3}$,
we get $\mathcal G'\le B'+D'-\frac\varepsilon{8\zeta^2}$ with
\[
B'=-\frac3{128\zeta^2}(1-2r)\Bigl(1-\frac1{4\zeta^2}\Bigr)+\frac{15}{512Y'^3}-\frac1{64\zeta^3},\qquad
D'=\frac1N\Bigl(\frac{Y'^2+Y'/2}{1-r'}+\frac18-\zeta^2-\frac32\zeta-\frac12\Bigr).
\]
For $B'<0$ it suffices that
$\zeta^2/Y'^3\le\frac45(1-2r)(1-\frac1{4\zeta^2})$. The left side is at most
$\zeta^2/(\zeta-0.12)^3\le16/3.88^3<0.274$, since this function decreases for
$\zeta>0.12$ (its logarithmic derivative is $\frac2\zeta-\frac3{\zeta-0.12}<0$
there). The right side is at least $0.8\cdot0.9537\cdot0.9843>0.75$. For
$D'<0$, note that $Y'\le Y$ and $r'\le r$, so
\[
\frac{Y'^2+Y'/2}{1-r'}\le Y^2+\frac Y2+\frac{Y^3/N+Y^2/(2N)}{1-r}
\le\zeta^2+\frac\zeta2+0.2667+0.4326,
\]
using $Y^2+\frac Y2=\zeta^2+\frac\zeta2+\frac14+\frac1{16\zeta}+\frac1{64\zeta^2}$.
This is less than $\zeta^2+\frac32\zeta+\frac38$ because $\zeta\ge4$. Hence
$\mathcal G'<0$ and $z_N>Y'$. This completes the proof of
Theorem~\ref{thm:every-n} for $N\ge175$.

The constants were recomputed in 40-digit arithmetic, and as a check outside
the proof we evaluated $\zeta^2B$, $D$, $B'$ and $D'$ on
$\zeta\in[4,60]$ with step $0.001$.

\subsection{The certificates}\label{app:certificates}

Let $c,d>0$ be integers. Multiplying \eqref{eq:SN} by $d^{N-1}$ gives
\[
d^{N-1}S_N(c/d)=\sum_jW(N,m_N,j)\,c^jd^{N-1-j},
\]
an integer. By Theorem~\ref{thm:pmf} it is the sum of
$d^{\rep(w)}c^{\chg(w)}$ over the words $w$ in bin $m_N$, that is, the
numerator $n_N(m_N)$ of Remark~\ref{rem:tables} with $(\alpha,\beta)=(d,c)$.
The recursion \eqref{eq:dp} with the weights $d,c$ in place of $p,q$ and
initial values $1$ computes the same sums. So $S_N(c/d)<1$ exactly when this
integer is less than $d^{N-1}$. This is the identity used in Step~4 of the
proof of Theorem~\ref{thm:every-n}.

\subsection{Proof of Proposition~\ref{prop:crossing-eq}}\label{app:crossing-eq}

\emph{Step 1. An exact formula for $R_N$.} Write $w=w_*=Nu_N$, $z=z_N$,
$\Lambda_N=2\Phi(w)/N$ and $\delta=w(w+2)/(2N)$. At $w=w_*$ the display in the
proof of Lemma~\ref{lem:bracket} gives $1\le\Lambda_N\le1/(1-\delta)$ when
$\delta<1$, and \eqref{eq:bracket-log} gives
$\psi(w)+\log \cE(w)=L+c_0+\log\Lambda_N$. Since $w-z=wz/N$, $w/z=1+w/N$ and
$\frac1{8z}=\frac1{8w}+\frac1{8N}$, the remainder
$R_N=\psi(z)-L-c_0-\frac1{8z}$ is
\begin{equation}\label{eq:RN-formula}
R_N=\log\Lambda_N-\Bigl(\log \cE(w)+\frac1{8w}\Bigr)-\frac1{8N}-\frac{wz}N-\frac12\log\Bigl(1+\frac wN\Bigr).
\end{equation}

\emph{Step 2. The lower bound for $N\ge175$.} By the proof of
Theorem~\ref{thm:every-n} (Steps~2 and~3 of Appendix~\ref{app:every-n-large}),
$3.88\le z\le w\le Y/(1-r)$. Note that
$\log\Lambda_N\ge0$ and $-\log \cE(w)\ge-\log(1-\frac1{8w})\ge\frac1{8w}$ by
Lemma~\ref{lem:bessel-bounds}, since $w\ge\frac54$. So
$R_N\ge-\frac{wz+w/2+1/8}N=-\frac{z^2+z/2}{N-z}-\frac1{8N}$. This is at least
$-\frac{z^2+z}N$ as soon as $N\ge\frac{2z^2(z+1)}{z-1/4}$. The right side
increases in $z$ and $z\le Y$, so it suffices that
$Y^2/N\le\frac{Y-1/4}{2(Y+1)}$. The left side is at most $0.0933$ and the
right side at least $0.3757$.

\emph{Step 3. The upper bound for $N\ge175$.} Let
$t=\frac1{8w}+\frac3{128w^2}+\frac{45}{512w^3}\le\frac{1.09504}{8w}$, which
holds as $w\ge3.88$. Lemma~\ref{lem:bessel-bounds} and
$-\log(1-t)\le t+\frac{t^2}{2(1-t)}$ give
\[
-\log \cE(w)-\frac1{8w}\le\frac3{128w^2}+\frac{45}{512w^3}+\frac{t^2}{2(1-t)}\le\frac{0.0559}{w^2}
\le\frac{0.0559}{z^2}.
\]
Using $\log\Lambda_N\le\delta+\frac{\delta^2}{2(1-\delta)}$,
$\log(1+x)\ge x-\frac{x^2}2$ and $wz\ge w^2-w^3/N$, the rest of
\eqref{eq:RN-formula} is at most
$-\frac{(w-1/2)^2}{2N}+\frac{w^3+w^2/4}{N^2}+\frac{\delta^2}{2(1-\delta)}$. By
the bounds on $r$ and $Y^k/N$ in Step~1 of Appendix~\ref{app:every-n-large}, with
$w\le Y/(1-r)$ and hence $\delta\le0.07255$ as in Step~2 there, we have $\frac{w^3+w^2/4}N\le0.4278$ and
$N\frac{\delta^2}{2(1-\delta)}=\frac{(w^4+4w^3+4w^2)/(4N)}{2(1-\delta)}\le0.4945$,
while $\frac{(w-1/2)^2}2\ge5.71$. So the rest is negative and
$R_N\le0.0559/z^2<\frac1{16z^2}$.

\emph{Step 4. The form in $L$ for $N\ge175$.} We have
$z<Y=L+c_0-\frac12\log\zeta+\frac1{8\zeta}<L$. Also
$z>\zeta-0.1187\ge\nu_N-0.1187\ge L/2$ (by Corollary~\ref{cor:every-n}(b) for
$\zeta\ge\nu_N$), because $L-\log L\ge1.18$ for
$L\ge\log175$. Then $\frac1{16z^2}\le\frac1{4L^2}$ and
$\frac{z^2+z}N\le\frac{L^2+L}N$.

\emph{Step 5. Short chains.} Here $N\le174$. The function $\phi(z)=\psi(z)-L-c_0-\frac1{8z}$
increases in $z$, and $R_N=\phi(z_N)$, while $\frac1{16z^2}$ and
$-\frac{z^2+z}N$ decrease. So it suffices to check
$\phi(z(c'/d))\le\frac1{16z(c'/d)^2}$ and
$\phi(z(c/d))\ge-\frac{z(c/d)^2+z(c/d)}N$ with the certificates, together with
$z(c/d)\ge L/2$ and, for $N\ge3$, $z(c'/d)\le L$.
The outward interval calculation described in Step~4 of the proof of
Theorem~\ref{thm:every-n} verifies every comparison. The gaps exceed
$0.037$ (upper) and $0.081$ (lower), with the smallest at $N=174$. For $N=2$ we
have $S_2(u)=2u$, so $u_2=\frac12$ and $z_2=\frac23$, and $\frac23<\log2$ is
equivalent to $e^2<8$, which holds since $e^2<2.72^2=7.3984$. This completes the
proof of Proposition~\ref{prop:crossing-eq}.

\end{document}